\pdfoutput=1
\documentclass[11pt,oneside,reqno]{amsart}

\usepackage{amsmath, amssymb, amsthm, amsfonts, mathrsfs, mathtools}
			
\setbox0=\hbox{$+$}
\newdimen\plusheight
\plusheight=\ht0
\def\+{\;\lower\plusheight\hbox{$+$}\;}

\newdimen\minusheight
\minusheight=\ht0
\def\-{\;\lower\minusheight\hbox{$-$}\;}

\setbox0=\hbox{$\cdots$}
\newdimen\cdotsheight
\cdotsheight=\plusheight%\ht0
\def\cds{\lower\cdotsheight\hbox{$\cdots$}}
\usepackage{tikz}
\usetikzlibrary{matrix,decorations.pathreplacing}

\newtheorem{question}{Question}
\newtheorem{conjecture}{Conjecture}

\theoremstyle{definition}
\newtheorem{problem}{Problem}

\usepackage{graphicx, epsfig}
\theoremstyle{definition}
\makeatletter

\usepackage[strict]{changepage}
\makeatother

\numberwithin{equation}{section}
\theoremstyle{plain}
\newtheorem{theorem}{Theorem}[section]
\newtheorem{corollary}[theorem]{Corollary}

\newtheorem{definition}[theorem]{Definition}
\newtheorem{lemma}[theorem]{Lemma}

\usepackage{titlesec}

\usepackage{graphicx, epsfig}
\theoremstyle{definition}
\usepackage{graphicx} 
\usepackage{color} 
\usepackage{amsthm}
\usepackage{changepage}
\usepackage{blkarray}
\usepackage{transparent} 
\usepackage{titlesec}
\usepackage{changepage}
\usepackage{lipsum}
\titleformat{\section}{\LARGE\bfseries}
\titleformat{\subsubsection}{\large\bfseries}

\makeatletter
\renewcommand\section{\@startsection{section}{1}{\z@}%
                                  {-3.5ex \@plus -1ex \@minus -.2ex}%
                                 {2.3ex \@plus.2ex}%
                                 {\normalfont\Large\bfseries}}
\renewcommand\subsection{\@startsection{subsection}{1}{\z@}%
                                  {-3.5ex \@plus -1ex \@minus -.2ex}%
                                 {2.3ex \@plus.2ex}%
                                 {\normalfont\large\bfseries}}

\makeatother
\usepackage{graphicx} 
\usepackage{color} 
\usepackage{transparent} 
\usepackage{geometry}
\begin{document}
\title{Rigidity in Hyperbolic Dehn filling}
\author{BoGwang Jeon \\
with an appendix by Ian Agol}
\maketitle
\begin{abstract}
This paper concerns with a rigidity of core geodesics in hyperbolic Dehn fillings. For instance, for an $n$-cusped hyperbolic $3$-manifold $\mathcal{M}$ having non-symmetric cusp shapes, we show any Dehn filling of $\mathcal{M}$ with sufficiently large coefficient is uniquely determined by the product of the holonomies of its core geodesics. We also explore various implications of the main results. An appendix by I. Agol provides an alternative geometric proof of one of the corollaries of our main arguments.   
\end{abstract}
\section{Introduction}
\subsection{Main results}
The following is the main statement of Thurston's hyperbolic Dehn filling theory:
\begin{theorem}[Thurston]
Let $\mathcal{M}$ be an $n$-cusped hyperbolic $3$-manifold and $\mathcal{M}_{(p_1/q_1,\dots, p_n/q_n)}$ be its $(p_1/q_1, \dots, p_n/q_n)$-Dehn filling. Then $\mathcal{M}_{(p_1/q_1,\dots, p_n/q_n)}$ is hyperbolic for $|p_i|+|q_i|$ ($1\leq i\leq n$) sufficiently large. 
\end{theorem}
Topologically, $\mathcal{M}_{(p_1/q_1, \dots, p_n/q_n)}$ is obtained by adding $n$-geodesics to $\mathcal{M}$ so called the \textit{core geodesics} of $\mathcal{M}_{(p_1/q_1, \dots, p_n/q_n)}$. The following theorem was also proved by Thurston as a part of his hyperbolic Dehn filling theory:

\begin{theorem}[Thurston]\label{19092501}
Let $\mathcal{M}$ and $\mathcal{M}_{(p_1/q_1,\dots, p_n/q_n)}$ be the same as above. 
\begin{enumerate}
\item For $|p_i|+|q_i|$ ($1\leq i\leq n$) sufficiently large, the first $n$-shortest geodesics of  
$\mathcal{M}_{(p_1/q_1,\dots, p_n/q_n)}$ are its core geodesics.
\item The lengths of core geodesics all converge to $0$ as $|p_i|+|q_i|\rightarrow\infty$. 
\end{enumerate} 
\end{theorem}

%Thus the topology and geometry of $\mathcal{M}_{(p_1/q_1,\dots, p_n/q_n)}$ is governed by its core geodesics. 

In this paper, we prove a strong rigidity of core geodesics in hyperbolic Dehn fillings. Before stating the main results of the paper, we first introduce 
\begin{definition}
Let $\mathcal{M}$ be an $n$-cusped hyperbolic $3$-manifold with cusp shapes $\tau_1, \dots, \tau_n$. We say $\mathcal{M}$ has non-symmetric cusp shapes if, for any $i,j$ ($1\leq i\neq j\leq n$), there are no $a,b,c,d\in \mathbb{Z}$ such that   
\begin{equation*}
\tau_i= \frac{a\tau_j+b}{c\tau_j+d}
\end{equation*} 
\end{definition}
If $\mathcal{M}$ has non-symmetric cusp shapes, then there is no self-isometry of $\mathcal{M}$ sending one cusp to another. Also note that the above definition implies each $\tau_i$ is non-quadratic. 

Now we state the first main result of this paper:
\begin{theorem}\label{19072401}
Let $\mathcal{M}$ be an $n$-cusped hyperbolic $3$-manifold having non-symmetric cusp shapes. Let 
\begin{equation*}
\big\{t^1_{(p_1/q_1,\dots, p_n/q_n)}, \dots, t^n_{(p_1/q_1, \dots, p_n/q_n)}\big\}\quad (\text{where }|t^i_{(p_1/q_1,\dots, p_n/q_n)}|>1)
\end{equation*} 
be the set of holonomies of $(p_1/q_1, \dots, p_n/q_n)$-Dehn filling of $\mathcal{M}$. For sufficiently large $|p_i|+|q_i|$ and $|p'_i|+|q'_i|$ ($1\leq i\leq n$), if 
\begin{equation}\label{19092902}
\prod^n_{i=1}\Big(t^{j_i}_{(p_1/q_1,\dots, p_n/q_n)}\Big)^{a_i}=\prod^n_{i=1}\Big(t^{j'_i}_{(p'_1/q'_1,\dots, p'_n/q'_n)}\Big)^{a'_i},
\end{equation}
then 
\begin{equation*}
\Big(t^{j_i}_{(p_1/q_1,\dots, p_n/q_n)}\Big)^{a_i}=\Big(t^{j'_i}_{(p'_1/q'_1,\dots, p'_n/q'_n)}\Big)^{a'_i},
\end{equation*}
for each $1\leq i\leq n$. Moreover, if $a_i\neq 0$, then  
%\begin{equation*}
%t^{i}_{(p_1/q_1,\dots, p_n/q_n)}=t^{i}_{(p'_1/q'_1,\dots, p'_n/q'_n)}
%\end{equation*} 
\begin{equation*}
p_i/q_i=p'_i/q'_i.
\end{equation*}
\end{theorem}

When $a_i=a'_i=1$ for all $i$, the above theorem implies 
\begin{corollary}\label{19092801}
Let $\mathcal{M}$ and $t^i_{(p_1/q_1,\dots, p_n/q_n)}$ be the same as above. For sufficiently large $|p_i|+|q_i|$ and $|p'_i|+|q'_i|$ ($1\leq i\leq n$), 
\begin{equation*}
\prod^{n}_{i=1}t^{i}_{(p_1/q_1,\dots, p_n/q_n)}=\prod^n_{i=1}t^{i }_{(p'_1/q'_1,\dots, p'_n/q'_n)}
\end{equation*}
 iff
\begin{equation*}
(p_1/q_1,\dots, p_n/q_n)=(p'_1/q'_1,\dots, p'_n/q'_n).
\end{equation*}
\end{corollary}

Corollary \ref{19092801} implies
\begin{corollary}\label{19080902}
Let $\mathcal{M}$ and $t^i_{(p_1/q_1,\dots, p_n/q_n)}$ be the same as above. For sufficiently large $|p_i|+|q_i|$ and $|p'_i|+|q'_i|$ ($1\leq i\leq n$), 
\begin{equation*}
\big\{t^{1}_{(p_1/q_1,\dots, p_n/q_n)}, \dots, t^{n}_{(p_1/q_1,\dots, p_n/q_n)}\big\}=\big\{t^{1}_{(p'_1/q'_1,\dots, p'_n/q'_n)}, \dots, t^{n}_{(p'_1/q'_1,\dots, p'_n/q'_n)}\big\}
\end{equation*}
 iff
\begin{equation*}
(p_1/q_1,\dots, p_n/q_n)=(p'_1/q'_1,\dots, p'_n/q'_n).
\end{equation*}
\end{corollary}

As mentioned earlier, I. Agol will give a simple geometric proof of the above corollary in the appendix. See also \cite{FPS}.

Combining with Theorem \ref{19092501}, Corollary \ref{19080902} induces
\begin{corollary}
Let $\mathcal{M}$ be the same as above. For sufficiently large $|p_i|+|q_i|$ and $|p'_i|+|q'_i|$ ($1\leq i\leq n$), 
\begin{equation*}
\mathcal{M}_{(p_1/q_1,\dots, p_n/q_n)}\cong\mathcal{M}_{(p'_1/q'_1,\dots, p'_n/q'_n)}
\end{equation*} 
iff
\begin{equation*}
(p_1/q_1,\dots, p_n/q_n)=(p'_1/q'_1,\dots, p'_n/q'_n).
\end{equation*}
\end{corollary}

\subsection{Complex volume}
\subsubsection{A question of Gromov}
The following theorem is due to Thurston:
\begin{theorem}[Thurston]\label{19100202}
Let $\mathcal{M}$ be an $n$-cusped hyperbolic $3$-manifold. Then 
\begin{equation*}
\text{vol}\;\mathcal{M}_{(p_1/q_1, \dots, p_n/q_n)}<\text{vol}\;\mathcal{M}
\end{equation*}
for any $(p_1/q_1, \dots, p_n/q_n)$ and 
\begin{equation*}
\text{vol}\;\mathcal{M}_{(p_1/q_1, \dots, p_n/q_n)}\rightarrow \text{vol}\;\mathcal{M}
\end{equation*}
as $|p_i|+|q_i| \rightarrow \infty$ ($1\leq i\leq n$). 
\end{theorem}

As a corollary, it follows that the number $N_{\mathcal{M}}(v)$ of hyperbolic Dehn fillings of $\mathcal{M}$ having the same volume $v$ is always finite. In 1979, Gromov raised the following question in his Bourbaki note \cite{gromov}:

\begin{question}[Gromov]\label{19100201}
For a given $n$-cusped hyperbolic $3$-manifold $\mathcal{M}$, does there exist $c$ such that 
\begin{equation*}
N_{\mathcal{M}}(v)<c
\end{equation*}
for any $v\in \mathbb{R}$?
\end{question}
Recently, the author answered the question partially as follows:
\begin{theorem}\cite{jeon3}
Let $\mathcal{M}$ be a $1$-cusped hyperbolic $3$-manifold whose cusp shape is quadratic. Then there exists $c$ depending only on $\mathcal{M}$ such that 
\begin{equation*}
N_{\mathcal{M}}(v)<c
\end{equation*}
for any $v\in \mathbb{R}$. 
\end{theorem}
The \textit{volume} and \textit{Chern-Simons} (CS) invariant of $\mathcal{M}$ are often considered simultaneously as the imaginary and real parts of the same complex number, so one defines the \textit{complex volume} of $\mathcal{M}$ as\footnote{Note that CS $\mathcal{M}$ is well defined modulo $\pi^2\mathbb{Z}$.} 
\begin{equation*}
\text{vol}_{\mathbb{C}}\;\mathcal{M}:=\text{vol}\;\mathcal{M}+i\; \text{CS}\;\mathcal{M} \mod i\pi^2\mathbb{Z}.
\end{equation*}
Let $N_{\mathcal{M}}^{\mathbb{C}}(v)$ be the number of Dehn fillings of $\mathcal{M}$ whose complex volumes are equal to $v\in \mathbb{C}$ (modulo $i\pi^2\mathbb{Z}$). Then one may ask the following natural analogue of Question \ref{19100201}. 
 
\begin{question}
Let $\mathcal{M}$ be an $n$-cusped hyperbolic $3$-manifold. Does there exist $c$ such that 
\begin{equation*}
N_{\mathcal{M}}^{\mathbb{C}}(v)<c
\end{equation*}
for any $v\in \mathbb{C}$?
\end{question}

A quantitative approach to Theorem \ref{19100202} was obtained by Neumann-Zagier in \cite{nz}, and later it was further generalized by T. Yoshida \cite{Y}.
\begin{theorem}[Neumann-Zagier, Yoshida]\label{19102301}
Let $\mathcal{M}$ be an $n$-cusped hyperbolic $3$-manifold. Then 
\begin{equation*}
\text{vol}_{\mathbb{C}}\;\mathcal{M}_{(p_1/q_1, \dots, p_n/q_n)}\equiv\text{vol}_{\mathbb{C}}\;\mathcal{M}-\frac{\pi}{2}\log\Big(\prod_{i=1}^n t^i_{(p_1/q_1,\dots, p_n/q_n)}\Big)+\epsilon(p_1/q_1, \dots, p_n/q_n)\quad \mod i\pi^2\mathbb{Z}
\end{equation*}
where $\epsilon(p_1/q_1, \dots, p_n/q_n)\rightarrow0$ as $|p_i|+|q_i|\rightarrow \infty$ ($1\leq i\leq n$). 
\end{theorem}
By the above theorem, $\log\Big(\prod_{i=1}^n t^i_{(p_1/q_1,\dots, p_n/q_n)}\Big)$ is the majority part of the complex volume of $\mathcal{M}$. Now we define the \textit{pseudo complex volume} of $\mathcal{M}_{(p_1/q_1, \dots, p_n/q_n)}$ as
\begin{equation*}
\text{pvol}_{\mathbb{C}}\;\mathcal{M}_{(p_1/q_1, \dots, p_n/q_n)}:=\text{vol}_{\mathbb{C}}\;\mathcal{M}-\frac{\pi}{2}\log\Big(\prod_{i=1}^n t^i_{(p_1/q_1,\dots, p_n/q_n)}\Big)\mod i\pi^2\mathbb{Z}.
\end{equation*} 
Then Corollary \ref{19092801} immediately implies
\begin{theorem}
Let $\mathcal{M}$ be an $n$-cusped hyperbolic $3$-manifold having non-symmetric cusp shapes. Then any Dehn filling of $\mathcal{M}$ with sufficiently large coefficient is uniquely determined by its pseudo complex volume. 
\end{theorem}

The following conjecture is now plausible:
\begin{conjecture}
Let $\mathcal{M}$ be an $n$-cusped hyperbolic $3$-manifold having non-symmetric cusp shapes. Then 
\begin{equation*}
N_{\mathcal{M}}^{\mathbb{C}}(v)=1
\end{equation*} 
for any $v\in \mathbb{C}$ sufficiently close to $\text{vol}_{\mathbb{C}}\;\mathcal{M}$ (modulo $i\pi^2\mathbb{Z}$). 
\end{conjecture}
%This conjecture implies, generically, a Dehn filling of $\mathcal{M}$ is uniquely determined by its complex volume. 

\subsubsection{Linear independence of pseudo complex volumes over $\mathbb{Q}$}
If a hyperbolic $3$-manifold $\mathcal{M}$ is a covering space of $\mathcal{N}$, then it is well-known that $\text{vol}\;\mathcal{M}$ is a constant multiple of $\text{vol}\;\mathcal{N}$. More generally, if two hyperbolic $3$-manifolds belong to the same commensurability class, then their volumes are rationally related. 

The following problem was proposed by Thurston in \cite{thu2} and has been remained unsolved:  
\begin{problem}[Thurston]
Show that volumes of hyperbolic 3-manifolds are not all rationally related.
\end{problem}

We consider an analogous problem for the pseudo complex volume case and get the following theorem. This is the second main result of the paper. 
\begin{theorem}\label{19090606}
Let $\mathcal{M}$ be a $1$-cusped hyperbolic $3$-manifold having non-quadratic cusp shape and $\mathcal{M}_{p/q}$ be its $p/q$-Dehn filling. For $n>0$, if
\begin{equation*}
\{p_1/q_1, \dots, p_n/q_n\}
\end{equation*}
is $n$ distinct Dehn filling coefficients with $|p_i|+|q_i|$ ($1\leq i\leq n$) sufficiently large, then there are no $a_i\in \mathbb{Q}\backslash\{0\}$ ($1\leq i\leq n$) such that 
\begin{equation*}
a_1\text{pvol}_{\mathbb{C}}\;\mathcal{M}_{p_1/q_1}+\cdots+a_n\text{pvol}_{\mathbb{C}}\;\mathcal{M}_{p_n/q_n}\equiv 0\mod i\pi^2\mathbb{Z}.
\end{equation*}  
\end{theorem}

Inspired from this, we propose the following conjecture:
\begin{conjecture}
Let $\mathcal{M}$ be an $1$-cusped hyperbolic $3$-manifold having non-quadratic cusp shapes. For any $n>0$, the volumes of $n$ distinct Dehn fillings of $\mathcal{M}$ with sufficiently large coefficients are linearly independent over $\mathbb{Q}$. 
\end{conjecture}

Finally we have the following definition and theorem. 
\begin{definition}
We say that $\alpha_1, \dots, \alpha_n\in \overline{\mathbb{Q}}$ are multiplicatively independent if 
\begin{equation*}
\alpha_1^{a_1}\cdots \alpha_n^{a_n}\neq 1
\end{equation*}
for any $(a_1, \dots, a_n)\in \mathbb{Z}^n\backslash\{0, \dots, 0\}$.
\end{definition}

\begin{theorem}\label{19100207}
Let $\mathcal{M}$ be an $n$-cusped hyperbolic $3$-manifold having non-symmetric cusp shapes. Then, for sufficiently large $|p_i|+|q_i|$ ($1\leq i\leq n$), the holonomies of $\mathcal{M}_{(p_1/q_1,\dots, p_n/q_n)}$ are multiplicatively independent. 
\end{theorem}

\subsection{Key idea}

The basic strategy is similar to the one attempted in \cite{jeon4}. Let $\mathcal{M}$ be an $n$-cusped hyperbolic $3$-manifold and $\mathcal{X}$ be its A-polynomial. If there exist two Dehn fillings of $\mathcal{M}$ satisfying \eqref{19092902}, there exists a nontrivial intersection point between $\mathcal{X}\times \mathcal{X}(\subset \mathbb{C}^{4n})$ and an algebraic subgroup of codimension $2n+1$. Thus if there are infinitely many such pairs of hyperbolic Dehn fillings, we get infinitely many algebraic subgroup $\{H_i\}_{i\in \mathcal{I}}$ of codimension $2n+1$ satisfying 
\begin{equation}\label{19101203}
(\mathcal{X}\times \mathcal{X})\cap H_i\neq \emptyset\quad (i\in \mathcal{I}).
\end{equation} 
In \cite{jeon4}, we assumed the Zilber-pink conjecture to get the main arguments there. But in this paper, instead of the Zilber-Pink conjecture, we use a slightly weaker conjecture so called the \textit{torsion openess conjecture} formulated by Bombieri-Masser-Zannier in \cite{za}. We prove the conjecture for a special case, locally near a small neighborhood of the identity for the A-polynomial of any cusped hyperbolic $3$-manifold. Since most Dehn filling points are distributed near a small neighborhood of the identity, this partial solution suffices for our purpose.  

It was proved by the author that the height of any point \eqref{19101203} is uniformly bounded (see Theorem \ref{19101204}). Also, by the work of Bombieri-Masser-Zannier, the degree of any \eqref{19101203} is uniformly bounded except for those points lying over torsion anomalous subvarieties of $\mathcal{X}\times \mathcal{X}$ (see Theorem \ref{19090804}). Since there are only finitely many algebraic number of bounded height and degree, we get, possibly except for finitely many, almost all of
\begin{equation}\label{19101301}
\bigcup_{i\in \mathcal{I}}(\mathcal{X}\times \mathcal{X})\cap H_i
\end{equation} 
are contained in arbitrarily many torsion anomalous subvarieties of $\mathcal{X}\times \mathcal{X}$. Now thanks to the aforementioned partial solution of the torsion openess conjecture, locally near the identity, the union of infinitely many torsion anomalous subvarieties of $\mathcal{X}\times\mathcal{X}$ are contained in only a finite number of maximal torsion anomalous subvarieties of $\mathcal{X}\times \mathcal{X}$ (see Theorem \ref{19090806}). 

In conclusion, \eqref{19101301} is contained in the union of only finitely many torsion anomalous subvarieties of $\mathcal{X}\times \mathcal{X}$. Finally, we get the desired result by analyzing the structure of each maximal torsion anomalous subvariety of $\mathcal{X}\times \mathcal{X}$.

\subsection{Outline of the paper}
In Section \ref{19100204}, we study some necessary background. We go over basic concepts and theorems required in later sections, but no new results will be proved in this section. In Section \ref{19100206}, we provide preliminary lemmas and theorems needed in the proofs the main theorems. In particular, we state the torsion openess conjecture and prove it for a special case in Section \ref{19100205}. Section \ref{19100209} is the heart of the paper and Theorems \ref{19072401} and \ref{19100207} will be shown in this section. In fact, we formulate and prove a stronger statement, and Theorems \ref{19072401} and \ref{19100207} then naturally follow as corollaries of this stronger one. Finally, we prove Theorem \ref{19090606} in Section \ref{19101404}.

%But we encourage audience to skip this section at first reading without seeing any of the proofs of these lemmas.  
\section{Background}\label{19100204}
\subsection{A-polynomial}
For a given $n$-cusped hyperbolic $3$-manifold $\mathcal{M}$, let $T_i$ be a torus cross-section of the $i$\textsuperscript{th}-cusp of $\mathcal{M}$ and $l_i, m_i$ be the chosen longitude-meridian pair of $T_i$ $(1\leq i\leq n)$. The A-polynomial of $\mathcal{M}$ is defined as
\begin{equation}\label{19071901}
\text{Hom}\;\big(\pi_1(\mathcal{M}), SL_2\mathbb{C}\big)/\sim,
\end{equation}
with the holonomies of $l_i, m_i$ $(1\leq i\leq n)$ parameters. We typically denote the holonomies of $l_i, m_i$ by $M_i, L_i$ $(1\leq i\leq n)$. It is known that \eqref{19071901} is an $n$-dimensional algebraic variety in $\mathbb{C}^{2n}(:=(M_1, L_1, \dots, M_n, L_n))$ and $(1,\dots, 1)$ gives rise to the complete hyperbolic metric structure of $\mathcal{M}$. 
\subsection{Dehn Filling}\label{Dehn}   
Let $\mathcal{M}$ be the same as above and 
\begin{equation*}
\mathcal{M}_{(p_1/q_1,...,p_n/q_n)}
\end{equation*}
be the $(p_1/q_1,\dots,p_n/q_n)$-Dehn filling of $\mathcal{M}$. By the Seifert-Van Kampen theorem, the fundamental group of $\mathcal{M}_{(p_1/q_1,...,p_n/q_n)}$ is obtained by adding the relations
\begin{gather*}
m_1^{p_1}l_1^{q_1}=1, \quad...\quad, m_n^{p_n}l_n^{q_n}=1
\end{gather*}
to the fundamental group of $\mathcal{M}$ where $m_i, l_i$ ($1\leq i\leq n$) are meridian-longitude pairs as previously. Hence if $\mathcal{X}$ is the A-polynomial of $\mathcal{M}$, then the discrete faithful representation 
\begin{equation*}
\phi\;:\;\pi_1(\mathcal{M}_{(p_1/q_1, \dots, p_n/q_n)})\longrightarrow SL_2\mathbb{C}
\end{equation*}
corresponds to an intersection point between $\mathcal{X}$ and an algebraic subgroup defined by  
\begin{equation} \label{Dehn eq}
M_1^{p_1}L_1^{q_1}=1,\quad ... \quad , M_n^{p_n}L_n^{q_n}=1. 
\end{equation}
We call \eqref{Dehn eq} the \emph{Dehn filling equations} with coefficient $(p_1/q_1,\dots,p_n/q_n)$ and a point inducing the hyperbolic structure on $\mathcal{M}_{(p_1/q_1,...,p_n/q_n)}$ a \emph{Dehn filling point} corresponding to $\mathcal{M}_{(p_1/q_1,...,p_n/q_n)}$. Let 
\begin{equation}\label{040401}
(M_1,L_1,\dots,M_n,L_n)=(t_1^{-q_1}, t_1^{p_1},\dots, t_n^{-q_n},t_n^{p_n})
\end{equation}
where $|t_i|>1$ for each $i$ ($1\leq i\leq n$) be a one of the Dehn filling points corresponding to $\mathcal{M}_{(p_1/q_1,...,p_n/q_n)}$. Then the holonomy of each core geodesic $m_i^{s_i}l_i^{r_i}$ where $p_ir_i-q_is_i=1$ is 
\begin{equation}\label{19101601}
M_i^{s_i}L_i^{r_i}=(t_i^{-q_i})^{s_i}(t_i^{p_i})^{r_i}=t_i^{-q_is_i+p_is_i}=t_i\quad (1\leq i\leq n). 
\end{equation}
We define 
\begin{equation*}
\{t_1, \dots, t_n\}
\end{equation*}
as the set of \emph{holonomies} of \eqref{040401} or $\mathcal{M}_{(p_1/q_1,...,p_n/q_n)}$.

The following theorem is a part of Thurston's hyperbolic Dehn filling theory \cite{nz}\cite{thu}. 
\begin{theorem} \label{040605}
Using the same notation as above, 
\begin{equation*}
(t_1^{-q_1}, t_1^{p_1},\dots, t_n^{-q_n},t_n^{p_n})\rightarrow(1,\dots,1)\in \mathbb{C}^{2n}
\end{equation*}
and 
\begin{equation*}
(t_1,\dots,t_n)\rightarrow(1,\dots,1)\in\mathbb{C}^{n}
\end{equation*}
as $|p_i|+|q_i|\rightarrow\infty$ for $1\leq i\leq n$. Moreover, 
\begin{equation*}
|t_i|\neq 1
\end{equation*} 
for every $1\leq i\leq n$ . 
\end{theorem}  
By Theorem \ref{040605}, most Dehn filling points are lying on a small neighborhood  $\mathcal{N}_\mathcal{X}$ of $(1, \dots, 1)$ in $\mathcal{X}$. If
\begin{align} 
u_i:=\log M_i,\quad v_i:=\log L_i\quad (1\leq i\leq n) \label{mer}
\end{align}
where $(M_1, L_1, \dots, M_n, L_n)\in \mathcal{N}_\mathcal{X}$, then the following statements hold in a neighborhood of the origin in $\mathbb{C}^n$ with $u_1,\dots,u_n$ as coordinates \cite{nz}.
\begin{theorem} \label{potential}
\begin{enumerate}
\item For each $i$ ($1\leq i\leq n$), there exists a holomorphic function $\tau _i(u_1,\dots,u_n)$ such that $\tau_i(0,\dots,0)=\tau_i\in \mathbb{C}\backslash\mathbb{R}$ and 
\begin{equation*}
v_i=u_i\cdot\tau_i(u_1,\dots,u_n).
\end{equation*}
\item There is a holomorphic function $\Phi(u_1,\dots,u_n)$ such that $\Phi(u_1,\dots,u_n)$ is even in each argument, $\Phi(0,\dots,0)=0$ and 
\begin{equation*}
v_i=\dfrac{1}{2}\dfrac{\partial \Phi}{\partial u_i}  
\end{equation*}
for each $i$ $(1\leq i\leq n)$.
%\begin{equation*}
%\Phi(u_1,\dots,u_n)=(\tau_1u_1^2+\cdots+\tau_n u_n^2)+(m_{4,\dots,0}u_1^4+\cdots+m_{0,\dots,4}u_n^4)+\emph{(higher order)}.\\
%\end{equation*}
\end{enumerate}
\end{theorem}
Thus the Taylor expansions of $\Phi(u_1, \dots, u_n), v_i(u_1, \dots, u_n)$ are of the following forms:
\begin{equation}\label{19100301}
\begin{gathered}
\Phi(u_1, \dots, u_n)=\sum_{j=1}^{n}\tau_ju_j^2+\sum_{\substack{k=4,\\k:even}}^{\infty}\sum_{\substack{j_1+\cdots+j_n=k,\\
j_l:even(1\leq l\leq n)}}c_{j_1,\dots,j_n}u_1^{j_1}\cdots u_n^{j_n},\\
v_i(u_1, \dots, u_n)=\tau_iu_i+\frac{1}{2}\sum_{\substack{k=4, \\ k:even}}^{\infty}\sum_{\substack{j_1+\cdots+j_n=k,\\ j_l: even(1\leq l\leq n)}}j_ic_{j_1,\dots,j_n}u_1^{j_1}\cdots u_i^{j_i-1}\cdots u_n^{j_n}
\end{gathered} 
\end{equation} 
respectively. We call $\tau_i$ the \textit{cusp shape} of $T_i$ with respect to $l_i, m_i$ and $\Phi(u_1,\dots,u_n)$ the \emph{Neumann-Zagier potential function} of $\mathcal{M}$ with respect to $l_i, m_i$ $(1\leq i\leq n)$. The complex manifold defined by
\begin{equation}\label{19101103}
v_i=u_i\cdot\tau_i(u_1,\dots,u_n)\quad (1\leq i\leq n)
\end{equation}
in $\mathbb{C}^{2n}(:=(u_1,v_1, \dots, u_n, v_n))$ is denoted by $\log \mathcal{N}_{\mathcal{X}}$. Clearly $\mathcal{N}_{\mathcal{X}}$ and $\log \mathcal{N}_{\mathcal{X}}$ are locally biholomorphic to each other and we call $\mathcal{N}_{\mathcal{X}}$ the \textit{Dehn filling} subset of $\mathcal{X}$. 

\subsection{Generalized A-polynomial}
We extend the class of A-polynomials. Let $\mathcal{X}$ be an A-polynomial of some $n$-cusped hyperbolic $3$-manifold ($n\geq 2$) and $H$ be an algebraic subgroup defined by 
\begin{equation}\label{19082001}
M_1^{a_1}L_1^{b_1}\cdots M_n^{a_n}L_n^{b_n}=1.
\end{equation}
Suppose $(a_i, b_i)\neq (0,0)$ and let 
\begin{equation}\label{19082001}
\text{Pr}_i:\;(M_1, L_1, \dots, M_n,L_n)\longrightarrow (M_1, L_1, \dots, M_{i-1}, L_{i-1}, M_{i+1}, L_{i+1}, \dots, M_n,L_n)
\end{equation}
be a projection from $\mathbb{C}^{2n}$ to $\mathbb{C}^{2n-2}$. We denote the image of $\mathcal{X}\cap H$ under \eqref{19082001} by $\text{Pr}_i(\mathcal{X}\cap H)$ and the algebraic closure of $\text{Pr}_i(\mathcal{X}\cap H)$ by $\overline{\text{Pr}_i(\mathcal{X}\cap H)}$.\footnote{Note that $\overline{\text{Pr}_i(\mathcal{X}\cap H)}$ is an $(n-1)$-dimensional algebraic variety in $(2n-2)$-dimensional algebraic variety.} Without loss of generality, let $i=1$ in \eqref{19082001}. If the holomorphic representations of $\log\mathcal{N}_\mathcal{X}$ are given as in \eqref{19100301}, then $\log (\mathcal{N}_\mathcal{X}\cap H)$ is the intersection between \eqref{19100301} and
\begin{equation}\label{19082202}
\begin{gathered}
a_1u_1+b_1v_1+\cdots+a_nu_n+b_nv_n=0.
\end{gathered}
\end{equation}
Since \eqref{19082202} is equivalent to   
\begin{equation}\label{19082002}
u_1=\sum_{i=2}^{n}-\frac{a_i+b_i\tau_i}{a_1+b_1\tau_1}u_i+\text{higher degrees},
\end{equation} 
by substituting \eqref{19082002} into \eqref{19100301}, we get the holomorphic parametrizations of $\log\big(\text{Pr}_1(\mathcal{N}_\mathcal{X}\cap H)\big)$
in $\mathbb{C}^{2n-2}(:=(u_2, v_2, \dots, u_n, v_n))$ where each $v_i$ ($2\leq i\leq n$) is of the form  
\begin{equation}\label{19100303}
v_i(u_2, \dots, u_n)=\tau_iu_i+\text{higher degrees}\quad (2\leq i\leq n).
\end{equation} 
By abuse of notation, we call 
\begin{equation*}
\tau_2, \dots, \tau_n
\end{equation*}
the \textit{cusp shapes} of $\overline{\text{Pr}_1(\mathcal{X}\cap H)}$ and $\text{Pr}_1(\mathcal{N}_\mathcal{X}\cap H)$ the \textit{Dehn filling} subset of $\overline{\text{Pr}_1(\mathcal{X}\cap H)}$.  

Motivated by this, we now extend the class of A-polynomials inductively as follows:
\begin{definition}
Let $\chi_n^{(0)}$ ($n\geq 1$) be the set of all A-polynomials of $n$-cusped hyperbolic $3$-manifolds. For $\mathcal{X}\in \chi_{n}^{(k)}$ ($k\geq 0$), suppose 
\begin{equation}\label{19082502}
\mathcal{X}\subset \mathbb{C}^{2n}(:=(M_{i_1},L_{i_1},\cdots M_{i_{n}},L_{i_{n}}))
\end{equation}
where the coordinates in \eqref{19082502} are the same as in Section \ref{Dehn} for $k=0$ and obtained canonically by induction as below for $k\geq 1$. For $n\geq 2$, we define $\chi_{n-1}^{(k+1)}$ to be the set of all the algebraic closure of   
\begin{equation*}
\text{Pr}_j(\mathcal{X}\cap H)
\end{equation*}
where $\mathcal{X}\in \chi_{n}^{(k)}$, $H$ is an algebraic subgroup of codimension $1$ defined by 
\begin{equation}\label{19082001}
M_{i_1}^{a_1}L_{i_1}^{b_1}\cdots M_{i_{n}}^{a_{n}}L_{i_{n}}^{b_{n}}=1, \quad (a_{i_j}, b_{i_j})\neq (0,0)\quad (1\leq j\leq n), 
\end{equation}
and 
\begin{equation*}
\text{Pr}_j:\;(M_{i_1}, L_{i_1}, \dots, M_{i_{n}},L_{i_{n}})\longrightarrow (M_{i_1}, L_{i_1}, \dots, M_{i_{j-1}}, L_{i_{j-1}}, M_{i_{j+1}}, L_{i_{j+1}}, \dots, M_{i_{n}},L_{i_{n}})
\end{equation*}
is a projection from $\mathbb{C}^{2n}$ to $\mathbb{C}^{2n-2}$. An element in
\begin{equation*}
\bigcup_{n=1}^{\infty} \bigcup_{k=0}^{n-1}\chi_n^{(k)} 
\end{equation*}
is called a generalized A-polynomial. 
\end{definition}
Note that $\mathcal{X}\in \chi_n^{(k)}$ is an $n$-dimensional variety and it is the projected image of the intersection between the A-polynomial $\mathcal{Z}$ of an $(n+k)$-cusped hyperbolic $3$-manifold and an algebraic subgroup $H$ of coimension $k$. Let $\mathcal{N}_{\mathcal{Z}}$ be the Dehn filling subset of $\mathcal{Z}$. By abuse of notation, we call the projected image of $\mathcal{N}_{\mathcal{Z}}\cap H$ in $\mathcal{X}$ the \textit{Dehn filling subset} of $\mathcal{X}$ and denote it by $\mathcal{N}_{\mathcal{X}}$. In fact, $\mathcal{X}$ possess similar properties to A-polynomials of $n$-cusped hyperbolic $3$-manifolds. For instance, there are canonical coordinates $M_{i_j}, L_{i_j}$ ($1\leq  j\leq n$) such that if 
\begin{equation*}
u_{ij}:=\log M_{i_j}, \quad v_{ij}:=\log L_{i_j} \quad (1\leq j\leq n)
\end{equation*} 
for $(M_{i_1}, L_{i_1}, \dots, M_{i_n}, L_{i_n})\in \mathcal{N}_{\mathcal{X}}$, we get holomorphic representations of the following forms
\begin{equation*}
v_j=\tau_ju_j+\text{higher degrees}\quad (i_1\leq j\leq i_{n})
\end{equation*} 
with $\tau_j\in \mathbb{C}\backslash\mathbb{R}$. Similar to A-polynomials, we call 
\begin{equation}\label{19100401}
\tau_{i_1}, \dots, \tau_{i_{n}}
\end{equation}
the \textit{cusp shapes} of $\mathcal{X}$.\footnote{Note that \eqref{19100401} is a subset of the cusp shapes of $\mathcal{Z}$.} 

If there exists an intersection point $P$ between $\mathcal{N}_{\mathcal{X}}$ and an algebraic subgroup defined by
\begin{equation*}
M_{i_1}^{p_1}L_{i_1}^{q_1}=1, \dots, M_{i_{n}}^{p_{n}}L_{i_{n}}^{q_{n}}=1,
\end{equation*}
then $P$ is of the following form:
\begin{equation}\label{19100403}
(t_1^{-q_1}, t_1^{p_1}, \dots, t_n^{-q_n}, t_n^{p_n})
\end{equation}
for some $t_1, \dots, t_n\in \mathbb{C}$. By abuse of notation once again, if \eqref{19100403} is non-elliptic (i.e. $|t_i|\neq 1$ for all $i$), we call \eqref{19100403} a \textit{Dehn filling} point of $\mathcal{X}$ associated with filling coefficient $(p_1/q_1, \dots, p_n/q_n)$ and 
\begin{equation}
\{t_1, \dots, t_n\}
\end{equation}
the set of the \textit{holomomies} of \eqref{19100403}.
  
Now Theorem \ref{19090803} can be generalized as follows:
\begin{theorem}\label{19101204}\cite{jeon1}\cite{jeon2}
For any generalized A-polynomial $\mathcal{X}$, the height of any Dehn filling point of it is uniformly bounded by some constant depending only on $\mathcal{X}$.  
\end{theorem}

\subsection{Height} \label{height1}
The height $h(\alpha)$ of an algebraic number $\alpha$ is defined as follows:
\begin{definition} 
Let $K$ be an any number field containing $\alpha$, $V_K$ be the set of places of $K$, and 
$K_v, \mathbb{Q}_v$ be the completions at $v\in V_K$. Then 
\begin{equation*}
h(\alpha):=\sum_{v \in M_K} \log\;\big(\emph{max} \{1, |\alpha|_v\}\big)^{[K_v:\mathbb{Q}_v]/[K:\mathbb{Q}]}.
\end{equation*} 
\end{definition}
Note that the above definition does not depend on the choice of $K$. 
That is, for any number field $K$ containing $\alpha$, it gives us the same value. 
If $\alpha=(\alpha_1,...,\alpha_n) \in \mathbb{\overline{Q}}^{n}$ is an $n$-tuple of algebraic numbers, the definition can be generalized as follows: 
\begin{definition} \label{height2}
Let $K$ be an any number field containing $\alpha_1,...,\alpha_n$, $V_K$ be the set of places of $K$, and 
$K_v, \mathbb{Q}_v$ be the completions at $v$. Then
\begin{equation*}
h(\alpha)=\sum_{v \in V_K} \log\;\big(\emph{max} \{1, |\alpha_1|_v,...,|\alpha_n|_v\}\big)^{[K_v:\mathbb{Q}_v]/[K:\mathbb{Q}]}.
\end{equation*} 
\end{definition}

The following theorem due to D. Northcott.   
\begin{theorem}[D. Northcott]\label{height}
There are only finitely many algebraic numbers of bounded height and degree.
\end{theorem}

In \cite{jeon1} \cite{jeon2}, the author proved the following theorem:
\begin{theorem}\label{19090803}
Let $\mathcal{M}$ be a cusped hyperbolic $3$-manifold and $\mathcal{X}$ be its A-polynomial. Then the height of any Dehn filling point of $\mathcal{X}$ is uniformly bounded. 
\end{theorem}

\subsection{Anomalous Subvarieties} \label{BHC1}

In this section, we identify $G^n$ with the non-vanishing of the coordinates $x_1,\dots,x_n$ in the affine $n$-space $\overline {\mathbb{Q}}^n$ or $\mathbb{C}^n$ (i.e.~$G^n=(\overline{\mathbb{Q}}^*)^n$ or $(\mathbb{C^*})^n$). 
An algebraic subgroup $H_{\Lambda}$ of $G^n$ is defined as the set of solutions satisfying equations $x_1^{a_1}\cdots x_n^{a_n}=1$ where the vector $(a_1,\dots,a_n)$ runs through a lattice $\Lambda \subset \mathbb{Z}^n$. 
If $\Lambda$ is primitive, then we call $H_{\Lambda}$ an irreducible algebraic subgroup or algebraic torus. 
By a coset $K$, we mean a translate $gH$ of some algebraic subgroup $H$ by some $g \in G^n$. 
For more properties of algebraic subgroups and $G^n$, see \cite{bg}.
\begin{definition} \label{anomalous1} 
An irreducible subvariety $\mathcal{Y}$ of $\mathcal{X}$ is anomalous if it has positive dimension and lies in a coset $K$ in $G^n$ satisfying
\begin{equation*}
\emph{dim } \mathcal{X} + \emph{dim } K -n+1\leq \emph{dim } \mathcal{Y} .
\end{equation*}
If $K$ is an algebraic subgroup, then we call $\mathcal{Y}$ a torsion anomalous subvariety of $\mathcal{X}$. 
\end{definition}

The quantity $\text{dim } \mathcal{X}+\text{dim }K-n$ is what one would expect for the dimension of $\mathcal{X}\cap K$ when $\mathcal{X}$ and $K$ were in general position. Thus one can understand anomalous subvarieties of $\mathcal{X}$ as the ones that are unusually large intersections with cosets of algebraic subgroups of $G^n$.

\begin{definition} \label{oa}
The deprived set $\mathcal{X}^{oa}$ (resp. $\mathcal{X}^{ta}$) is what remains of $\mathcal{X}$ after removing all anomalous subvarieties (resp. all torsion anomalous subvarieties).
\end{definition}

\begin{definition} \label{maxano}
An anomalous subvariety of $\mathcal{X}$ is maximal if it is not contained in a strictly larger anomalous subvariety of $\mathcal{X}$.
\end{definition}

The following theorem due to Bombieri-Masser-Zannier tells us the structure of anomalous subvarieties \cite{za}.

\begin{theorem}\label{struc}
Let $\mathcal{X}$ be an irreducible variety in $G^n$ of positive dimension defined over $\mathbb{\overline Q}$. \\
(a) For any torus $H$ with 
\begin{equation} \label{dim1}
1\leq n- \emph{dim } H \leq \emph{dim } \mathcal{X},
\end{equation}
the union $Z_H$ of all subvarieties $\mathcal{Y}$ of $\mathcal{X}$ contained in any coset $K$ of $H$ with
\begin{equation} \label{dim2}
\emph{dim } H=n-(1+\emph{dim } \mathcal{X})+\emph{dim } \mathcal{Y} 
\end{equation}
is a closed subset of $\mathcal{X}$, and the product $HZ_H$ is not Zariski dense in $G^n$.\\
(b) There is a finite collection $\Psi=\Psi_{\mathcal{X}}$ of such tori $H$ such that every maximal anomalous subvariety $\mathcal{Y}$ of $\mathcal{X}$ is a component of $\mathcal{X} \cap gH$
for some $H$ in $\Psi$ satisfying \eqref{dim1} and \eqref{dim2} and some $g$ in $Z_H$. Moreover $\mathcal{X}^{oa}$ is obtained from $\mathcal{X}$ by removing the $Z_H$ of all $H$ in $\Psi$, 
and thus it is open in $\mathcal{X}$ with respect to the Zariski topology.
\end{theorem} 

Now we restrict our attention to A-polynomials of hyperbolic $3$-manifolds. Let $\mathcal{X}$ be the A-polynomial of an $n$-cusped hyperbolic $3$-manifold. If $\log \mathcal{N}_{\mathcal{X}}$ is defined by \eqref{19101103}, then, for each $i$ ($1\leq i\leq n$),
\begin{equation*}
u_i=0 \quad \Longleftrightarrow \quad  v_i=0.
\end{equation*}
Thus $M_i=1$ is equivalent to $L_i=1$ over $\mathcal{X}$, implying
\begin{equation*}
\mathcal{X}\cap (M_i=L_i=1)
\end{equation*}
is an anomalous subvariety of $\mathcal{X}$. 

Since we are only interested in $\mathcal{N}_{\mathcal{X}}$, we further refine Definition \ref{anomalous1} as follows:
\begin{definition} 
Let $\mathcal{X}$ be a generalized A-polynomial in $G^n$ and $\mathcal{N}_{\mathcal{X}}$ be its Dehn filling subset. For an algebraic coset $K$, we say $\mathcal{N}_{\mathcal{X}}$ intersects $K$ anomalously (or $\mathcal{N}_{\mathcal{X}}\cap K$ is an anomalous subset of $\mathcal{N}_{\mathcal{X}}$) if
\begin{equation*}
\emph{dim } \mathcal{N}_{\mathcal{X}} + \emph{dim } K -n+1\leq \emph{dim } (\mathcal{N}_{\mathcal{X}}\cap K) .
\end{equation*}
\end{definition}

\section{Preliminary Results}\label{19100206}
In this section, we prove several lemmas and theorems will play key roles in the proofs of the main theorems. 

\subsection{Torsion openess conjecture}\label{19100205}
As mentioned in Introduction, the following conjecture was proposed by Bombieri-Masser-Zannier \cite{za}:
\begin{conjecture}[Torsion openess conjecture]\label{19101108}
Let $\mathcal{X}$ be an algebraic variety and $\{H_i\}_{i\in \mathcal{I}}$ be the set of algebraic subgroups such that $\mathcal{X}\cap H_i$ is an anomalous subvariety of $\mathcal{X}$. Then there exist a finite number of algebraic subgroups $K_1, \dots, K_m$ such that 
\begin{equation*}
\bigcup_{i\in \mathcal{I}}\mathcal{X}\cap H_i\subset \bigcup_{j=1}^m\mathcal{X}\cap K_j
\end{equation*}
\end{conjecture}
We resolve this conjecture locally near the identity for the A-polynomial of any cusped hyperbolic $3$-manifold. 
\begin{theorem}\label{19090805}
Let $\mathcal{X}$ be a generalized A-polynomial such that 
\begin{equation*}
\mathcal{X}\subset \mathbb{C}^{2n}(:=(M_1, L_1, \dots, M_n, L_n)).
\end{equation*} 
Let $\{H_i\}_{i\in\mathcal{I}}$ be the set of algebraic subgroups intersecting $\mathcal{N}_{\mathcal{X}}$ anomalously. Then there exists a finite number of algebraic subgroups $K_1, \dots, K_m$ such that
\begin{equation*}
\bigcup_{i\in \mathcal{I}}\mathcal{N}_{\mathcal{X}}\cap H_i\subset \bigcup_{j=1}^m \mathcal{N}_{\mathcal{X}}\cap K_j.
\end{equation*}
\end{theorem}
\begin{proof}[Proof of the theorem]
Let $H\in \{H_i\}_{i\in \mathcal{I}}$ be defined by 
\begin{equation*}
\begin{gathered}
M_1^{a_{11}}L_1^{b_{11}}\cdots M_n^{a_{1n}}L_n^{b_{1n}}=1,\\
\cdots \\
M_1^{a_{k1}}L_1^{b_{k1}}\cdots M_n^{a_{kn}}L_n^{b_{kn}}=1.
\end{gathered}
\end{equation*}
Then $\mathcal{N}_{\mathcal{X}}\cap H$ is locally biholomorphic to the complex manifold defined by
\begin{equation}\label{19071602}
\begin{gathered}
a_{11}u_1+b_{11}v_1+\cdots +a_{1n}u_n+b_{1n}v_n=2\pi i m_{1},\\
\cdots \\
a_{k1}u_1+b_{k1}v_1+\cdots +a_{kn}u_n+b_{kn}v_n=2\pi i m_{k}
\end{gathered}
\end{equation}
where $m_{1}, \dots, m_{k}\in \mathbb{Z}$. For $m_{i}\neq0$, dividing the $i$\textsuperscript{th} equation of \eqref{19071602} by $m_{i}$ if necessary, we assume the equations in \eqref{19071602} are normalized as 
\begin{equation}\label{19071603}
\begin{gathered}
a'_{11}u_1+b'_{11}v_1+\cdots +a'_{1n}u_n+b'_{1n}v_n=2\pi i \epsilon_{1},\\
\cdots \\
a'_{k1}u_1+b'_{k1}v_1+\cdots +a'_{kn}u_n+b'_{kn}v_n=2\pi i \epsilon_{k}
\end{gathered}
\end{equation}
where $\epsilon_{i}\in \{0, 1\}$ and $a'_{ij}, b'_{ij}\in \mathbb{Q}$. 

If $\epsilon_i=0$ for all $i$, then $\mathcal{N}_{\mathcal{X}}\cap H$ is an anomalous subset of $\mathcal{N}_{\mathcal{X}}$ containing $(1, \dots, 1)$. By Theorem \ref{struc}, since there are only finitely many maximal anomalous subvarieties of $\mathcal{X}$ containing $(1, \dots, 1)$, we get the desired result.  

Otherwise, without loss of generality, we assume $\epsilon_i=1$ for $1\leq i\leq l$ and $\epsilon_i=0$ for $l+1\leq i\leq k$ . Then \eqref{19071603} is equivalent to 
\begin{equation}\label{19071604}
\begin{gathered}
a'_{11}u_1+b'_{11}v_1+\cdots +a'_{1n}u_n+b'_{1n}v_n=2\pi i,\\
a'_{i1}u_1+b'_{i1}v_1+\cdots +a'_{in}u_n+b'_{in}v_n=a'_{11}u_1+b'_{11}v_1+\cdots +a'_{1n}u_n+b'_{1n}v_n \quad (\text{for }2\leq i\leq l),\\
a'_{i1}u_1+b'_{i1}v_1+\cdots +a'_{in}u_n+b'_{in}v_n=0\quad (\text{for } l+1\leq i\leq k).
\end{gathered}
\end{equation}
Let $H'$ be an algebraic subgroup of codimension $k-1$ defined by
\begin{equation*}
\begin{gathered}
M_1^{c_i(a'_{i1}-a'_{11})}L_1^{c_i(b'_{i1}-b'_{11})}\cdots M_n^{c_i(a'_{in}-a'_{1n})}L_n^{c_i(b'_{in}-b'_{1n})}=1\quad (\text{for }2\leq i\leq l),\\
M_1^{c_ia'_{i1}}L_1^{c_ib'_{i1}}\cdots M_n^{c_ia'_{in}}L_n^{c_ib'_{in}}=1\quad (\text{for } l+1\leq i\leq k)
\end{gathered}
\end{equation*}
where $c_i$ is an integer such that 
\begin{equation*}
\begin{gathered}
c_i(a'_{i1}-a'_{11}, b'_{i1}-b'_{11},\dots, a'_{in}-a'_{1n}, b'_{in}-b'_{1n})\in \mathbb{Z}^{2n}\quad (2\leq i\leq l),\\
c_i(a'_{i1}, b'_{i1},\dots a'_{in}, b'_{in})\in \mathbb{Z}^{2n}\quad (l+1\leq i\leq k).
\end{gathered}
\end{equation*}
Then $\log(\mathcal{N}_{\mathcal{X}}\cap H')$ is the complex manifold defined by 
\begin{equation}
\begin{gathered}
(a'_{i1}-a'_{11})u_1+(b'_{i1}-b'_{11})v_1+\cdots +(a'_{in}-a'_{1n})u_n+(b'_{in}-b'_{1n})v_n=0 \quad (2\leq i\leq l),\\
a'_{i1}u_1+b'_{i1}v_1+\cdots +a'_{in}u_n+b'_{in}v_n=0\quad (l+1\leq i\leq k).
\end{gathered}
\end{equation}
Since $\mathcal{N}_{\mathcal{X}}\cap H$ is an anomalous subset of $\mathcal{N}_{\mathcal{X}}$, the dimension of $\log(\mathcal{N}_{\mathcal{X}}\cap H)$ defined by \eqref{19071604} strictly bigger than $n-k$. Since $\log(\mathcal{N}_{\mathcal{X}}\cap H')$ contains $(0,\dots, 0)$, we have
\begin{equation*}
\text{dim}\;\big(\log(\mathcal{N}_{\mathcal{X}}\cap H')\big)=\text{dim}\;\big(\log(\mathcal{N}_{\mathcal{X}}\cap H)\big)+1>n-k+1.
\end{equation*}
In conclusion, $\mathcal{N}_{\mathcal{X}}\cap H'$ is also an anomalous subset of $\mathcal{N}_{\mathcal{X}}$ containing $(1,\dots, 1)$. By Theorem \ref{struc}, there are only finitely many maximal anomalous subsets of $\mathcal{N}_{\mathcal{X}}$ containing $(1,\dots, 1)$ and so we get the desired result.
\end{proof}

The following theorem was proved by Bombieri-Masser-Zannier in \cite{BMZ0}:

\begin{theorem}\label{19090804}
Let $\mathcal{X}$ be a variety in $G^n$ of dimension $k\leq n-1$ defined over $\overline{\mathbb{Q}}$. Then for any $B\geq 0$ there are at most finitely many points $P$ in $\mathcal{X}^{ta}(\overline{\mathbb{Q}})\cap \mathcal{H}_{n-k-1}$ with $h(P)\leq B$ where $\mathcal{H}_{n-k-1}$ is the set of algebraic subgroups of dimension $n-k-1$ (or codimension $k+1$). 
\end{theorem}

Combining the above theorem with Theorems \ref{19090803} and \ref{19090805}, we finally get the following theorem, which will be crucial in the proofs of the main theorems. 

\begin{theorem}\label{19090806}
Let $\mathcal{X}$ be a generalized A-polynomial and $\{P_i\}_{i\in\mathbb{N}}$ be a set of Dehn filling points in $\mathcal{N}_{\mathcal{X}}$ such that 
\begin{equation*}
\lim_{i\rightarrow \infty}P_i=(1, \dots, 1).
\end{equation*}
If each $P_i$ is contained in an algebraic subgroup of codimension $n+1$, then there exist a finite number of algebraic subgroups $K_1, \dots, K_m$ such that 
\begin{equation*}
P_i\in \bigcup_{j=1}^m \mathcal{N}_{\mathcal{X}}\cap K_j
\end{equation*} 
for all $P_i$ except for possibly finitely many. 
\begin{proof}
By Theorem \ref{19090803}, the height of $P_i$ is uniformly bounded. By Theorem \ref{19090804}, possibly except for finitely many, most $P_i$ are contained in torsion anomalous subsets of $\mathcal{N}_{\mathcal{X}}$. Now the conclusion follows from Theorem \ref{19090805}. 
\end{proof}
\end{theorem}

\subsection{Miscellaneous Lemmas}
In this subsection, we prove two lemmas required for the proof of the main theorems. The proofs of these lemmas are purely computational and elementary, and so a trusting reader can skip ahead at first reading.
\begin{lemma}\label{17052301}
Let 
\begin{gather}\label{17052201}
\left( \begin{array}{cccc}
a_1 & b_1 & c_1 & d_1  \\
a_2 & b_2 & c_2 & d_2 
\end{array} 
\right)
\end{gather}
be an integer matrix of rank $2$, and $\tau$ be a non-quadratic number. If the rank of the following $(2\times 2)$-matrix  
\begin{equation}\label{17052202}
\left( \begin{array}{cc}
a_1+b_1\tau & c_1+d_1\tau  \\
a_2+b_2\tau & c_2+d_2\tau 
\end{array} \right)
\end{equation}
is equal to $1$, then \eqref{17052201} is either of the following forms:
\begin{gather}\label{17052205}
\left( \begin{array}{cccc}
a_1 & b_1 & ma_1 & mb_1  \\
a_2 & b_2 & ma_2 & mb_2 
\end{array}\right)
\end{gather}
for some $m\in\mathbb{Q}$ or 
\begin{gather}\label{17052209}
\left( \begin{array}{cccc}
0 & 0 & c_1 & d_1  \\
0 & 0 & c_2 & d_2 
\end{array}\right).
\end{gather}
\end{lemma}
Note that since \eqref{17052201} is a matrix of rank $2$, $a_1b_2-a_2b_1\neq 0$ in \eqref{17052205} and $c_1d_2-c_2d_1\neq 0$ in \eqref{17052209}.
\begin{proof}
Since the rank of \eqref{17052202} is $1$, we have 
\begin{equation*}
(a_1+b_1\tau)(c_2+d_2\tau)=(a_2+b_2\tau)(c_1+d_1\tau).
\end{equation*}
By the assumption, $\tau$ is not quadratic, so 
\begin{gather}
a_1c_2=a_2c_1,\label{17052206}\\
b_1d_2=b_2d_1,\label{17052207}\\
b_1c_2+a_1d_2=a_2d_1+b_2c_1.\label{17052203}
\end{gather}
\begin{enumerate} 
\item If none of $a_i,b_i,c_i,d_i$ ($i=1,2$) are zero, then there exist $m,n\in\mathbb{Q}$ such that  
\begin{equation}\label{17052204}
\begin{gathered}
(a_1,a_2)=m(c_1,c_2),\quad (b_1,b_2)=n(d_1,d_2).
\end{gathered}
\end{equation}
By \eqref{17052203}, we get 
\begin{equation*}
nd_1c_2+mc_1d_2=mc_2d_1+nd_2c_1,
\end{equation*}
which is equivalent to 
\begin{equation*}
(n-m)(d_1c_2-d_2c_1)=0.
\end{equation*}
\begin{enumerate}
\item If $m=n$, then \eqref{17052201} is of the form given in \eqref{17052205}. 
\item If $d_1c_2-d_2c_1=0$, then there exists $l\in\mathbb{Q}$ such that $l(d_1,d_2)=(c_1,c_2)$. Combining with \eqref{17052204}, we get 
\begin{gather*}
(a_1,a_2)=m(c_1,c_2)=ml(d_1,d_2)=mln(b_1,b_2).
\end{gather*} 
This implies that $(a_1,b_1,c_1,d_1)=t(a_2,b_2,c_2,d_2)$ for some $t\in \mathbb{Q}$, and it contradicts the fact that \eqref{17052201} is a matrix of rank $2$. 
\end{enumerate}
\end{enumerate}
\quad\\
\item One of $a_i,b_i,c_i,d_i$ is equal to $0$. By symmetry, it is enough to consider the case $a_1=0$. If $a_1=0$, then $a_2=0$ or $c_1=0$ by \eqref{17052206}. \\
\begin{enumerate}
\item If $a_1=a_2=0$, then, by \eqref{17052207} and \eqref{17052203}, we get 
\begin{equation}\label{17052208}
b_1d_2=b_2d_1, \quad b_1c_2=b_2c_1.
\end{equation}
\begin{enumerate}
\item If none of $b_i, c_i, d_i$ are zero, then $(b_1,b_2)=m(c_1,c_2)$ and $(b_1,b_2)=n(d_1,d_2)$ for some $n,m\in \mathbb{Q}$. But this contradicts the fact \eqref{17052201} is a matrix of rank $2$. 
\item Suppose $b_1=0$ or $b_2=0$. Without loss of generality, we assume $b_1=0$. 
By \eqref{17052208}, if $b_2\neq 0$, then $c_1=d_1=0$ and so $a_1=b_1=c_1=d_1=0$. But this contradicts the fact that \eqref{17052201} is a matrix of rank $2$. 
Thus $b_2=0$ and the matrix \eqref{17052201} is of the following form:
\begin{gather*}
\left( \begin{array}{cccc}
0 & 0 & c_1 & d_1  \\
0 & 0 & c_2 & d_2 
\end{array} \right),
\end{gather*}
which is the case given in \eqref{17052209}. 
\item Suppose $b_1,b_2\neq 0$ and one of $c_i, d_i$ is zero. Without loss of generality, 
we consider $c_1=0$. Then, by \eqref{17052208}, $c_2=0$. 
Since $b_1d_2=d_1b_2$, we have $l(b_1,b_2)=(d_1,d_2)$ for some $l\in \mathbb{Q}$. But this contradicts the fact that the rank of \eqref{17052201} is $2$. 
\end{enumerate}
\quad \\
\item If $a_1=c_1=0$, then 
\begin{equation}\label{17052210}
b_1d_2=b_2d_1, \quad b_1c_2=a_2d_1.
\end{equation}
\begin{enumerate} 
\item If none of $b_i, d_i, c_2, d_2$ are zero, then $m(b_1,d_1)=(b_2,d_2)$ and $n(b_1,d_1)=(a_2,c_2)$ for some $m,n\in \mathbb{Q}$. So the matrix \eqref{17052201} is of the following form
\begin{gather*}
\left( \begin{array}{cccc}
0 & b_1 & 0 & d_1  \\
nb_1 & mb_1 & nd_1 & md_1 
\end{array} \right),
\end{gather*}
which is of the form given in \eqref{17052205}. 
\item If $b_1=0$, then $a_1=b_1=c_1=0$. Since \eqref{17052201} is a matrix of rank $2$, we assume $d_1\neq 0$. Then, by \eqref{17052210}, we have $a_2=b_2=0$ and so \eqref{17052201} is of the form given in \eqref{17052209}.  
\item If $d_1=0$, then $a_1=c_1=d_1=0$. Since \eqref{17052201} is a matrix of rank $2$, we assume $b_1\neq 0$. Then, by \eqref{17052210}, we have $c_2=d_2=0$ and so \eqref{17052201} is of the form given in \eqref{17052205}. 
\item If $b_2=0$ (with $b_1\neq 0$), then $d_2=0$ by \eqref{17052210}. 
Since we already dealt with the cases $a_2=0$ or $d_1=0$ above, we assume $a_2\neq 0$ and $d_1\neq 0$. By \eqref{17052210}, we have $m(b_1,d_1)=(a_2,c_2)$ for some $m\in \mathbb{Q}$, and so \eqref{17052201} is of the following form
\begin{gather*}
\left( \begin{array}{cccc}
0 & b_1 & 0 & d_1  \\
mb_1 & 0 & md_1 & 0 
\end{array} \right),
\end{gather*}
which is the case given in \eqref{17052205}. 
\item If $d_2=0$, then $b_2=0$ or $d_1=0$ by \eqref{17052210}. But we already considered these two cases above. 
\item If $c_2=0$, then $a_2=0$ or $d_1=0$ by \eqref{17052210}. We also covered these cases above.   
\end{enumerate}
\end{enumerate}
\end{proof}

\begin{lemma}\label{clai}
Let
\begin{gather}\label{17052401}
\left( \begin{array}{cccc}
a_1 & b_1 & c_1 & d_1  \\
a_2 & b_2 & c_2 & d_2 
\end{array} 
\right)
\end{gather}
be an integer matrix of rank $2$, and $\tau_1,\tau_2$ be algebraic numbers such that $1,\tau_1,\tau_2,\tau_1\tau_2$ are linearly independent over $\mathbb{Q}$. 
If the rank of the following $(2\times 2)$-matrix  
\begin{equation}\label{17052402}
\left( \begin{array}{cc}
a_1+b_1\tau_1 & c_1+d_1\tau_2  \\
a_2+b_2\tau_1 & c_2+d_2\tau_2 
\end{array} \right)
\end{equation}
is equal to $1$, then \eqref{17052401} is either 
\begin{gather}\label{17052501}
\left( \begin{array}{cccc}
a_1 & b_1 & 0 & 0  \\
a_2 & b_2 & 0 & 0 
\end{array}\right)
\end{gather}
or 
\begin{gather}\label{17052502}
\left( \begin{array}{cccc}
0 & 0 & c_1 & d_1  \\
0 & 0 & c_2 & d_2 
\end{array}\right).
\end{gather}
\end{lemma}

\begin{proof}
Since the rank of \eqref{17052402} is $1$, and as $1, \tau_1, \tau_2, \tau_1\tau_2$ are linearly independent over $\mathbb{Q}$, we have
\begin{align}
a_{1}c_{2}-c_{1}a_{2}=0, \label{17052403}\\
b_{1}c_{2}-c_{1}b_{2}=0, \label{17052404}\\
a_{1}d_{2}-d_{1}a_{2}=0, \label{17052405}\\
b_{1}d_{2}-d_{1}b_{2}=0. \label{17052406}
\end{align}
If none of $a_i, b_i, c_i, d_i$ ($i=1,2$) are zero, then \eqref{17052403}-\eqref{17052406} imply the two nonzero vectors $(a_1,b_1,c_1,d_1)$ and $(a_2,b_2,c_2,d_2)$ are linearly dependent over $\mathbb{Q}$. 
But this is impossible because \eqref{17052401} is a matrix of rank $2$. Without loss of generality, let us assume $a_1=0$. Then, by \eqref{17052403} and \eqref{17052405}, we have the following two cases:
\begin{enumerate}
\item $a_{2}=0$. 

In this case, the problem is reduced to the following:
\begin{gather}
b_{1}c_{2}-c_{1}b_{2}=0,\label{17052407}\\
b_{1}d_{2}-d_{1}b_{2}=0,\label{17052408}
\end{gather}
Similar to above, if none of $b_i,c_i,d_i$ ($i=1,2$) are zero, then $(b_1,c_1,d_1)$ and $(b_2,c_2,d_2)$ are linearly dependent over $\mathbb{Q}$ by \eqref{17052407} and \eqref{17052408}, contradicting the fact that \eqref{17052401} is a matrix of rank $2$. 
So at least one of $b_i,c_i,d_i$ ($i=1,2$) is zero and the situation is divided into the following two subcases. 
\begin{enumerate}
\item $b_1=0$ or $b_2=0$.

By symmetry, it is enough to consider the case $b_1=0$. If $b_1=0$, then $b_2=0$ or $c_1=0$ by \eqref{17052407} and $b_2=0$ or $d_1=0$ by \eqref{17052408}. 
If $b_2=0$, then we get the desired result (i.e.~$a_1=a_2=b_1=b_2=0$). Otherwise, if $c_1=d_1=0$, it contradicts the fact that $(a_1,b_1,c_1,d_1)$ is a nonzero vector.   

\item $c_1=0$ or $c_2=0$ or $d_1=0$ or $d_2=0$ (with $b_1,b_2\neq 0$). 

Here, also by symmetry, it is enough to consider the first case $c_1=0$. 
If $b_1,b_2\neq0$ and $c_1=0$, then $c_2=0$ by \eqref{17052407}, and $(d_1,d_2)=m(b_1,b_2)$ for some $m\in \mathbb{Q}$ by \eqref{17052408}. But this contradicts the fact that \eqref{17052401} is a matrix of rank $2$. 
\end{enumerate}
\item $a_2\neq0$ and so $c_1=d_1=0$. 

Since $(a_1,b_1,c_1,d_1)$ is a nonzero vector, $b_1$ is nonzero and $c_2=d_2=0$ by \eqref{17052404} and \eqref{17052406}. 
As a result, we get $c_1=c_2=d_1=d_2=0$, which is the second desired result of the statement. \end{enumerate}
So the lemma holds.
\end{proof}

\section{Main Result I}\label{19100209}
In this section, we prove Theorems \ref{19072401} and \ref{19100207}. As mentioned earlier, we formulate and prove a more general statement, Theorem \ref{19101107}, and then two theorems will follow as corollaries of this general one. 

Before proving Theorem \ref{19101107}, let us start with the following lemma. 
\begin{lemma}\label{19071803}
Let 
\begin{equation*}
\mathcal{X} \subset \mathbb{C}^{2n+2m}(:=(M_1, L_1, \dots, M_{n+m}, L_{n+m}))
\end{equation*}
be a generalized A-polynomial in $\chi_{n+m}^{(k)}$ and  
\begin{equation*}
\tau_{1},\dots, \tau_{n}, \tau_{n+1}(=\tau_1), \dots, \tau_{n+m}(=\tau_{m}) \quad (n\geq m)
\end{equation*}
be the cusp shapes of $\mathcal{X}$ with
\begin{equation*}
\tau_1, \dots, \tau_n
\end{equation*}
non-symmetric. If $H$ is an algebraic subgroup of codimension $2$ such that $\mathcal{N}_\mathcal{X}\cap H$ is an anomalous subset of $\mathcal{N}_\mathcal{X}$ containing $(1, \dots, 1)$, then $\mathcal{N}_\mathcal{X}\cap H$ is contained in either  
\begin{equation}
M_i=L_i=1\quad (1\leq i\leq n+m)
\end{equation}
or 
\begin{equation}
M_i^a=M_{n+i}^{b}, \quad L_i^a=L_{n+i}^b\quad (1\leq i\leq m)
\end{equation}
for some $a, b\in \mathbb{Z}$. 
\end{lemma}

\begin{proof}
Let $H$ be defined by
\begin{equation}\label{17052308}
\begin{split}
M_1^{a_{11}}L_1^{b_{11}}\cdots M_{n+m}^{a_{1(n+m)}}L_{n+m}^{b_{1(n+m)}}=1,\\
M_1^{a_{21}}L_1^{b_{21}}\cdots M_{n+m}^{a_{2(n+m)}}L_{n+m}^{b_{2(n+m)}}=1.
\end{split}
\end{equation}
If the holomorphic representations of $\log\mathcal{N}_{\mathcal{X}}$ are given as 
\begin{equation*}
v_i=\tau_iu_i+\text{higher degrees} \quad (1\leq i\leq n+m),
\end{equation*}
then $\mathcal{N}_\mathcal{X}\cap H$ is biholomorphic to the complex manifold defined by
\begin{equation}\label{17052101}
\begin{gathered}
a_{11}u_1+b_{11}(\tau_1u_1+\cdots)+\cdots+a_{1(n+m)}u_{n+m}+b_{1(n+m)}(\tau_{m}u_{n+m}+\cdots)=0,\\
a_{21}u_1+b_{21}(\tau_1u_1+\cdots)+\cdots+a_{2(n+m)}u_{n+m}+b_{2(n+m)}(\tau_{m}u_{n+m}+\cdots)=0.
\end{gathered}
\end{equation}
Since $\mathcal{N}_{\mathcal{X}}\cap H$ is an anomalous subset of $\mathcal{N}_{\mathcal{X}}$, \eqref{17052101} is a complex manifold of dimension $(n+m-1)$. Thus the rank of 
\begin{equation}\label{17052306}
\left(\begin{array}{ccc}
a_{11}+b_{11}\tau_1 & \cdots & a_{1(n+m)}+b_{1(n+m)}\tau_m\\
a_{21}+b_{21}\tau_1 & \cdots & a_{2(n+m)}+b_{2(n+m)}\tau_m\\
\end{array}\right),
\end{equation}
which is the Jacobian of \eqref{17052101} at $(0,\dots,0)$, is equal to $1$. 
\begin{enumerate}
\item Suppose 
\begin{equation}
\text{det}\left(\begin{array}{cc}
a_{1i} & b_{1i}\\
a_{2i} & b_{2i} \\
\end{array}\right)\neq 0
\end{equation}
for some $1\leq i\leq n+m$. We further spilt the problem into the following two subcases. 
\begin{enumerate}
\item First assume $1\leq i\leq m$ or $n+1\leq i\leq n+m$. Without loss of generality, let $i=1$. Since the rank of 
\begin{equation}
\left(\begin{array}{cc}
a_{11}+b_{11}\tau_1 & a_{1j}+b_{1j}\tau_j\\
a_{21}+b_{21}\tau_1 & a_{2j}+b_{2j}\tau_j\\
\end{array}\right)
\end{equation}
is $1$ for any $j$ ($1\leq j\leq n+m$), we get 
\begin{equation}
a_{1j}=b_{1j}=a_{2j}=b_{2j}=0\quad (j\neq 1, n+1)
\end{equation}
by Lemma \ref{clai} and 
\begin{equation}
\begin{gathered}
\left( \begin{array}{cc}
a_{1(n+1)} & b_{1(n+1)}\\
a_{2(n+1)} & b_{2(n+1)} 
\end{array}\right)
=\left( \begin{array}{cc}
la_{11} & lb_{11}  \\
la_{21} & lb_{21} 
\end{array}\right) 
\end{gathered}
\end{equation}
for some $l\in \mathbb{Q}$ by Lemma \ref{17052301}.

\item If $m+1\leq i\leq n$, since the rank of 
\begin{equation}
\left(\begin{array}{cc}
a_{1i}+b_{1i}\tau_i & a_{1j}+b_{1j}\tau_j\\
a_{2i}+b_{2i}\tau_i & a_{2j}+b_{2j}\tau_j\\
\end{array}\right)
\end{equation}
is $1$ for any $j$ ($1\leq j\leq n+m$), we get     
\begin{equation}
a_{1j}=b_{1j}=a_{2j}=b_{2j}=0
\end{equation}
for all $j(\neq i)$ by Lemma \ref{clai}. 
\end{enumerate}

\item Suppose   
\begin{equation}
\text{det}\left(\begin{array}{cc}
a_{1i} & b_{1i}\\
a_{2i} & b_{2i} \\
\end{array}\right)=0
\end{equation}
for every $i$ ($1\leq i\leq n+m$). Without loss of generality, suppose $a_{11}+\tau_1b_{11}\neq 0$ and $a_{21}=b_{21}=0$. If 
\begin{equation*}
a_{1i}=b_{1i}=a_{2i}=b_{2i}=0
\end{equation*} 
for all $2\leq i\leq n+m$, then it contradicts the fact that $H$ is an algebraic subgroup of codimension $2$. If there exists $i$ ($ 2\leq i\leq n+m$) such that 
\begin{equation}
\left(\begin{array}{c}
a_{1i}+\tau_ib_{1i}\\
a_{2i}+\tau_ib_{2i} \\
\end{array}\right)\neq 
\left(\begin{array}{c}
0\\
0\\
\end{array}\right),
\end{equation}
then $a_{2i}=b_{2i}=0$ by Lemmas \ref{17052301} and \ref{clai}, again contradicting the fact that $H$ is an algebraic subgroup of codimension $2$. \end{enumerate}
In conclusion, we get \eqref{17052308} is either one of the following forms:
\begin{equation}\label{19100801}
\begin{gathered}
M_i^{a_{1i}}L_i^{b_{1i}}=1\\
M_i^{a_{2i}}L_i^{b_{2i}}=1
\end{gathered}
\end{equation}
for some $1\leq i\leq n+m$ or 
\begin{equation}\label{19100802}
\begin{gathered}
M_i^{aa_{1i}}L_i^{ab_{1i}}M_{n+i}^{ba_{1i}}L_{n+i}^{bb_{1i}}=1,\\
M_i^{aa_{2i}}L_i^{ab_{2i}}M_{n+i}^{ba_{2i}}L_{n+i}^{bb_{2i}}=1
\end{gathered}
\end{equation}
for some $1\leq i\leq m$ and $a,b\in \mathbb{Z}$. Since $a_{1i}b_{2i}-a_{2i}b_{1i}\neq 0$, the components of \eqref{19100801} and \eqref{19100802} containing $(1, \dots, 1)$ are
\begin{equation}
\begin{gathered}
M_i=L_i=1 \quad (1\leq i\leq n+m)
\end{gathered}
\end{equation}
and 
\begin{equation}
\begin{gathered}
M_i^{a}M_{n+i}^{b}=L_i^{a}L_{n+i}^{b}=1\quad (1\leq i\leq m)
\end{gathered}
\end{equation}
respectively. This completes the proof. 
\end{proof}

{\bf Remark:} Note that, in the statement of Lemma \ref{19071803}, there are only finitely many choices for $a,b$ by Theorem \ref{struc}. (One can also prove this directly from   
\begin{equation*}
au_i=bu_{n+i} \Longleftrightarrow av_i=bv_{n+i}
\end{equation*}
using the representations of $v_i, v_{n+i}$.) We will use this fact in the proof of the main theorem below.

Now we prove the main statement. Note that Theorem \ref{19072401} (resp. Theorem \ref{19100207}) is obtained by letting $n=m$ (resp. $m=0$) in the following theorem.

\begin{theorem}\label{19101107}
Let $\mathcal{X}$ be the same as in Lemma \ref{19071803}. Let $P_i$ $(i\in \mathbb{N})$ be a $(p_{i1}/q_{i1},\dots, p_{i(n+m)}/q_{i(n+m)})$-Dehn filling point in $\mathcal{N}_{\mathcal{X}}$ satisfying\footnote{Note that \eqref{19102101} impllies  
\begin{equation*}
|p_{ij}|+|q_{ij}|\rightarrow \infty\quad \text{as}\quad i\rightarrow\infty
\end{equation*} 
for each $j$ ($1\leq j\leq n+m$).}
\begin{equation}\label{19102101}
P_i\rightarrow (1, \dots, 1)\quad \text{as}\quad i\rightarrow\infty
\end{equation}
and 
\begin{equation*}
\{t_{i1}, \dots, t_{i(n+m)}\}
\end{equation*} 
be the set of holonomies of $P_i$ ($i\in \mathbb{N}$). If 
\begin{equation}\label{19090201}
\prod_{j=1}^{n}(t_{ij})^{a_{ij}}=\prod_{j=n+1}^{n+m}(t_{ij})^{a_{ij}}\quad (a_{ij}\in \mathbb{Z}, i\in \mathbb{N}), 
\end{equation}
then, for $i$ sufficiently large, we have 
\begin{equation}\label{19090202}
\begin{gathered}
(t_{ij})^{a_{ij}}=(t_{i(n+j)})^{a_{i(n+j)}}\quad (1\leq j\leq m),\\
a_{ij}=0 \quad (m+1\leq j\leq n).
\end{gathered}
\end{equation}
Moreover, if $a_{ij}\neq 0$, then
\begin{equation*}
p_{ij}/q_{ij}=p_{i(n+j)}/q_{i(n+j)},
\end{equation*}
and there exists a finite subset $S_j\subset\mathbb{Q}$ depending only on $\mathcal{X}$ such that  
\begin{equation*}
a_{ij}/a_{i(n+j)}\in S_j.
\end{equation*} 
\end{theorem}

\begin{proof}[Proof of Theorem \ref{19101107}]
Recall $P_i$ is an intersection point between $\mathcal{X}$ and an algebraic subgroup defined by 
\begin{equation}\label{19101602}
\begin{gathered}
M_1^{p_{i1}}L_1^{q_{i1}}=1, \dots ,M_{n+m}^{p_{i(n+m)}}L_{n+m}^{q_{i(n+m)}}=1.
\end{gathered}
\end{equation}
By \eqref{19101601}, \eqref{19090201} corresponds to 
\begin{equation}\label{19071802}
\begin{gathered}
\prod_{j=1}^n(M_{j}^{r_{ij}}L_{j}^{s_{ij}})^{a_{ij}}=\prod_{j=n+1}^{n+m}(M_{j}^{r_{ij}}L_{j}^{s_{ij}})^{a_{ij}}
\end{gathered}
\end{equation}
where $p_{ij}s_{ij}-q_{ij}r_{ij}=1$ ($i\in \mathbb{N}$, $1\leq j\leq n+m$). So $P_i$ is an intersection point between $\mathcal{N}_{\mathcal{X}}$ and an algebraic subgroup of codimension $n+m+1$. By Theorem \ref{19090806}, there exists a finite number of algebraic subgroups $K_1, \dots, K_l$ such that 
\begin{equation*}
\mathcal{N}_{\mathcal{X}}\cap K_j\quad (1\leq j\leq l)
\end{equation*}
is an anomalous subset of $\mathcal{N}_{\mathcal{X}}$ and
\begin{equation*}
P_i\in \bigcup_{j=1}^l\mathcal{N}_{\mathcal{X}}\cap K_j
\end{equation*}
for all $i$ sufficiently large. Without loss of generality, suppose $l=1, K:=K_1$ and   
\begin{equation*}
\bigcup_{i\in \mathbb{N}}P_i\subset \mathcal{N}_{\mathcal{X}}\cap K.
\end{equation*}
We prove the theorem by induction on $n+m$.
\begin{enumerate}
\item If $n+m=2$, we have either $n=2, m=0$ or $n=m=1$. 
\begin{enumerate}
\item For $n=2, m=0$, by Lemma \ref{19071803}, $\mathcal{N}_{\mathcal{X}}\cap K$ is either contained in 
\begin{equation*}
M_1=L_1=1
\end{equation*}
or
\begin{equation*}
M_2=L_2=1.
\end{equation*}
But either contradicts the fact that Dehn filling points are non-elliptic. 

\item If $n=m=1$, again by Lemma \ref{19071803}, $\mathcal{N}_{\mathcal{X}}\cap K$ is contained in 
\begin{equation}
M_1^a=M_{2}^{b}, \quad L_1^a=L_{2}^b
\end{equation}
where $a, b\in \mathbb{Z}$. Thus  
\begin{equation*}
\big(t_{i1}^{-q_{i1}}\big)^{a}=\big(t_{i2}^{-q_{i2}}\big)^{b}, \quad \big(t_{i1}^{p_{i1}}\big)^{a}=\big(t_{i2}^{p_{i2}}\big)^{b},
\end{equation*}
implying
\begin{equation*}
\frac{q_{i2}b}{q_{i1}a}=\frac{p_{i2}b}{p_{i1}a}\Longrightarrow \frac{q_{i1}}{p_{i1}}=\frac{q_{i2}}{p_{i2}}
\end{equation*}
and 
\begin{equation}\label{19101101}
(t_{i1})^a=(t_{i2})^b
\end{equation}
for $i$ sufficiently large. Finally, if $a_{i1}, a_{i2}\neq 0$, we have 
\begin{equation*}
\frac{a_{i1}}{a_{i2}}=\frac{a}{b}
\end{equation*}  
by \eqref{19101101} (and \eqref{19090201}). Since there are only finitely many possibilities for $a,b$, we get the desired result. 
\end{enumerate}
\item Now suppose $n+m\geq 3$ and the theorem holds for $2, \dots, n+m-1$. 
\begin{enumerate}
\item If $K$ is of codimension $2$, then, by Lemma \ref{19071803}, we get $\mathcal{N}_{\mathcal{X}}\cap K$ is contained in 
\begin{equation}\label{19071807}
M_j^a=M_{n+j}^b, \quad L_j^a=L_{n+j}^b 
\end{equation}
for some $1\leq j\leq m$ and $a, b\in \mathbb{Z}$. Without loss of generality, suppose $j=1$ and $K$ is defined by
\begin{equation*}
M_1^a=M_{n+1}^b, \quad L_1^a=L_{n+1}^b.
\end{equation*}
As shown previously, this implies 
\begin{equation*}
\frac{q_{i1}}{p_{i1}}=\frac{q_{i(n+1)}}{p_{i(n+1)}}
\end{equation*}
and 
\begin{equation}\label{19100601}
(t_{i1})^a=(t_{i(n+1)})^b
\end{equation}
for $i$ sufficiently large. Let 
\begin{equation*}
\mathcal{X}_1:=\mathcal{X}\cap (M_1^a=M_{n+1}^b)
\end{equation*}
and 
\begin{equation*}
\begin{gathered}
\text{Pr}:\;(M_1, L_1, \dots, M_{n+m}, L_{n+m}) \longrightarrow (M_2, L_2, \dots, M_{n+m}, L_{n+m}).
\end{gathered}
\end{equation*}
Then $\overline{\text{Pr}\;\mathcal{X}_1}$ is a generalized A-polynomial in $ \chi_{n+m-1}^{(k+1)}$ with the cusp shapes
\begin{equation*}
\tau_2, \dots, \tau_n, \tau_{n+1}(=\tau_1), \dots, \tau_{n+m}(=\tau_{m}),
\end{equation*}
and $\text{Pr}\;P_i$ is a $(p_{i2}/q_{i2},\dots, p_{i(n+m)}/q_{i(n+m)})$-Dehn filling point of $\overline{\text{Pr}\;\mathcal{X}_1}$ with the set of the holonomies  
\begin{equation*}
\{t_{i2}, \dots, t_{i(n+m)}\}.
\end{equation*} 
By \eqref{19100601}, \eqref{19090201} is reduced to 
\begin{equation}
\begin{gathered}
\prod_{j=2}^{m}(t_{ij})^{a_{ij}}=(t_{i(n+1)})^{a_{i(n+1)}-\frac{b}{a}a_{i1}}\prod_{j=n+2}^{n+m}(t_{ij})^{a_{ij}} \quad (i\in \mathbb{N}).
\end{gathered}
\end{equation}
Since $\overline{\text{Pr}\;\mathcal{X}_1}\in\chi_{n+m-1}^{(k+1)}$, by induction, we get
\begin{equation}
\begin{gathered}
a_{i(n+1)}-\frac{b}{a}a_{i1}=0,\\
(t_{ij})^{a_{ij}}=(t_{i(n+j)})^{a_{i(n+j)}}\quad (2\leq j\leq m),\\
a_{ij}=0\quad (m+1\leq j\leq n)
\end{gathered}
\end{equation}
for $i$ sufficiently large. If $a_{i(n+1)}=a_{i1}=0$, then we are done. Otherwise, 
\begin{equation*}
a_{i(n+1)}-\frac{b}{a}a_{i1}=0 \Longrightarrow \frac{b}{a}=\frac{a_{i(n+1)}}{a_{i1}}.
\end{equation*}
Since there are only finitely many possibilities for $a$ and $b$, we get the desired result. 

\item Now suppose codimen $K\geq 3$. Applying Gauss elimination if necessary, we assume $K$ is defined by the following types of equations:
\begin{equation}\label{19062301}
\begin{gathered}
M_1^{a_{11}}L_1^{b_{11}}\cdots M_{n+m}^{a_{1(n+m)}}L_{n+m}^{b_{1(n+m)}}=1,\\
L_1^{b_{21}}\cdots M_{n+m}^{a_{2(n+m)}}L_{n+m}^{b_{2(n+m)}}=1,\\
M_2^{a_{32}}L_2^{b_{32}}\cdots M_{n+m}^{a_{3(n+m)}}L_{n+m}^{b_{3(n+m)}}=1,\\
\cdots.
\end{gathered}
\end{equation}
Without loss of generality, we further suppose $(a_{11}, b_{11})\neq (0,0)$.\footnote{If $a_{11}=b_{11}=0$, then we apply Gauss elimination again.} Let 
\begin{equation*}
\begin{gathered}
\mathcal{X}_1:=\mathcal{X}\cap (M_1^{a_{11}}L_1^{b_{11}}\cdots M_{n+m}^{a_{1(n+m)}}L_{n+m}^{b_{1(n+m)}}=1)
\end{gathered}
\end{equation*}
and 
\begin{equation*}
\begin{gathered}
\text{Pr}:\;(M_1, L_1, \dots, M_{n+m}, L_{n+m}) \longrightarrow (M_2, L_2, \dots, M_{n+m}, L_{n+m}).
\end{gathered}
\end{equation*} 
Then $\overline{\text{Pr}\;\mathcal{X}_1}$ is a generalized A-polynomial in $ \chi_{n+m-1}^{(k+1)}$ with cusp shapes
\begin{equation*}
\tau_2, \dots, \tau_n, \tau_{n+1}(=\tau_1), \dots, \tau_{n+m}(=\tau_{m}),
\end{equation*}
and $\text{Pr}\;P_i$ ($i\in \mathbb{N}$) is a $(p_{i2}/q_{i2},\dots, p_{i(n+m)}/q_{i(n+m)})$-Dehn filling point of $\overline{\text{Pr}\;\mathcal{X}_1}$ with the set of the holonomies   
\begin{equation*}
\{t_{i2}, \dots, t_{i(n+m)}\}.
\end{equation*} 
The third equation in \eqref{19062301} implies the holonomies of $\text{Pr}\;P_i$ are multiplicatively dependent, so, by induction, there exists some $j$ ($2\leq j\leq m$) and a finite set $S_j$ of rational numbers such that  
\begin{equation}\label{19090603}
\frac{p_{ij}}{q_{ij}}=\frac{p_{i(n+j)}}{q_{i(n+j)}}
\end{equation}
and 
\begin{equation}\label{19100604}
t_{ij}=(t_{i(n+j)})^r
\end{equation}
for all sufficiently large $i$ with some $r(=b/a)\in S_j$. Without loss of generality, suppose $j=2$. Then \eqref{19090603} and \eqref{19100604} imply
\begin{equation*}
(t_{i2})^{-q_{i2}}=\big((t_{i(n+2)})^{-q_{i(n+2)}}\big)^r, \quad (t_{i2})^{p_{i2}}=\big((t_{i(n+2)})^{p_{i(n+2)}}\big)^r
\end{equation*}
and so 
\begin{equation*}
P_i\in (M_2^a=M_{n+2}^b,L_2^a=L_{n+2}^{b})
\end{equation*}
for $i$ sufficiently large. Let
\begin{equation*}
\mathcal{X}_2:=\mathcal{X}\cap (M_2^a=M_{n+2}^b)
\end{equation*}
and
\begin{equation*}
\begin{gathered}
\text{Pr}:\;(M_1, L_1, \dots, M_{n+m}, L_{n+m}) \longrightarrow (M_1, L_1, M_3, L_3, \dots, M_{n+m}, L_{n+m}).
\end{gathered}
\end{equation*} 
Then $\overline{\text{Pr}\;\mathcal{X}_2}$ is a generalized A-polynomial in $\chi_{n+m-1}^{(k+1)}$ and $\text{Pr}\;P_i$ is a $(p_{i1}/q_{i1},p_{i3}/q_{i3},\dots, p_{i(n+m)}/q_{i(n+m)})$-Dehn filling point of $\overline{\text{Pr}\;\mathcal{X}_2}$ with   
\begin{equation*}
\{t_{i1},t_{i3}, \dots, t_{i(n+m)}\}
\end{equation*} 
the set of the holonomies of $\text{Pr}\;P_i$. By \eqref{19100604}, \eqref{19090201} is reduced to 
\begin{equation}
\begin{gathered}
(t_{i1})^{a_{i1}}\prod_{j=3}^{m}(t_{ij})^{a_{ij}}=(t_{i(n+1)})^{a_{i1}}(t_{i(n+2)})^{a_{i(n+2)}-ra_{i2}}\prod_{j=n+3}^{n+m}(t_{ij})^{a_{ij}}\quad (i\in \mathbb{N}). 
\end{gathered}
\end{equation}
By induction, 
\begin{equation}
\begin{gathered}
a_{i(n+2)}-ra_{i2}=0,\\
(t_{ij})^{a_{ij}}=(t_{i(n+j)})^{a_{i(n+j)}}\quad (j=1\;\text{or}\; 3\leq j\leq m),\\
a_{ij}=0\quad (m+1\leq j\leq n)
\end{gathered}
\end{equation}
for $i$ sufficiently large. If $a_{i(n+2)}=a_{i2}=0$, then we are done. Otherwise, 
\begin{equation*}
a_{i(n+2)}-ra_{i2}=0 \Longrightarrow r=\frac{a_{i(n+2)}}{a_{i2}}.
\end{equation*}
Since there are only finitely many possibilities for $r$, we get the desired result. 
\end{enumerate}
\end{enumerate}
\end{proof}

\section{Main Result II}\label{19101404}
This section is devoted to the proof of Theorem \ref{19090606}. The basic strategy of the proof is similar to that of the proof of Theorem \ref{19101107} in the previous section. 

Let $\mathcal{M}$ be a $1$-cusped hyperbolic $3$-manifold having non-quadratic cusp shape and $\mathcal{X}$ be its A-polynomial. For
\begin{equation*}
\mathcal{X}^{(n)}:=\mathcal{X}\times \cdots \times \mathcal{X}\subset \mathbb{C}^{2n}(:=(M_1, L_1, \dots, M_n, L_n)),
\end{equation*} 
which is the product of $n$ identical copies of $\mathcal{X}$, we consider it as the A-polynomial of an $n$-cusped hyperbolic $3$-manifold.

The following lemma is a refined version of Lemma \ref{19071803}. 
\begin{lemma}\label{19100901}
Let $\mathcal{X}^{(n)}$ be as above. Let $H$ is an algebraic subgroup of codimension $2$. If $\mathcal{N}_{\mathcal{X}^{(n)}}\cap H$ is an $(n-1)$-dimensional anomalous subset of $\mathcal{N}_{\mathcal{X}^{(n)}}$ containing $(1, \dots, 1)$, then $\mathcal{N}_{\mathcal{X}^{(n)}}\cap H$ is either contained in 
\begin{equation*}
M_i=L_i=1\quad (1\leq i\leq n)
\end{equation*}
or 
\begin{equation*}
M_i=M_j, \quad L_i=L_j\quad (1\leq i\neq j\leq n)
\end{equation*} 
or 
\begin{equation*}
M_i=M_j^{-1}, \quad L_i=L_j^{-1}\quad (1\leq i\neq j\leq n).
\end{equation*} 
\end{lemma}
Before proving the lemma, we cite the following, which is Lemma 2.3 in \cite{jeon3}.

\begin{lemma}\label{NZ}
Let $\mathcal{M}$ be a $1$-cusped hyperbolic $3$-manifold and $\mathcal{X}$ be its A-polynomial. If  
\begin{equation}\label{18110805}
v=c_1 u+c_{3}u^{3}+\cdots.
\end{equation}
is the holomorphic representation of $\log \mathcal{N}_{\mathcal{X}}$, then there exists $i$ ($i\geq 2$) such that $c_{2i-1}\neq 0$. In other words, \eqref{18110805} is not linear. 
\end{lemma}

\begin{proof}[Proof of Lemma \ref{19100901}]
Let $H$ be defined by 
\begin{equation*}
\begin{gathered}
M_1^{a_{11}}L_1^{b_{11}}\cdots M_n^{a_{1n}}L_n^{b_{1n}}=1,\\
M_1^{a_{21}}L_1^{b_{21}}\cdots M_n^{a_{2n}}L_n^{b_{2n}}=1.
\end{gathered}
\end{equation*} 
Since $\mathcal{X}^{(n)}$ is the product of $n$ copies of $\mathcal{X}$, the holomorphic representations of $\log\mathcal{N}_{\mathcal{X}^{(n)}}$ are 
\begin{equation*}
v_i=\tau u_i+h(u_i) \quad (1\leq i\leq n)
\end{equation*}
where $h$ is a non-trivial (by Lemma \ref{NZ}) holomorphic function of degree $\geq 3$. Now $\mathcal{N}_{\mathcal{X}^{(n)}}\cap H$ is biholomorphic to the complex manifold defined by
\begin{equation}\label{19100903}
\begin{gathered}
a_{11}u_1+b_{11}(\tau u_1+h(u_1))+\cdots+a_{1n}u_{n}+b_{1n}(\tau u_{n}+h(u_n))=0,\\
a_{21}u_1+b_{21}(\tau u_1+h(u_1))+\cdots+a_{2n}u_{n}+b_{2n}(\tau u_{n}+h(u_n))=0.
\end{gathered}
\end{equation}
Since $\mathcal{N}_{\mathcal{X}^{(n)}}\cap H$ is an $(n-1)$-dimensional anomalous subset of $\mathcal{N}_{\mathcal{X}^{(n)}}$, \eqref{19100903} is an $(n-1)$-dimensional complex manifold. Thus the rank of 
\begin{equation}
\left(\begin{array}{ccc}
a_{11}+b_{11}\tau & \cdots & a_{1n}+b_{1n}\tau\\
a_{21}+b_{21}\tau & \cdots & a_{2n}+b_{2n}\tau\\
\end{array}\right),
\end{equation}
which is the Jacobian of \eqref{19100903} at $(0,\dots,0)$, is equal to $1$. Without loss of generality, suppose\footnote{If \begin{equation}
\text{det}\left(\begin{array}{cc}
a_{1i} & b_{1i}\\
a_{2i} & b_{2i} \\
\end{array}\right)=0
\end{equation}
for all $1\leq i\leq n$, it contradicts the fact that $H$ is an algebraic subgroup of codimension $2$ (as seen in the proof of Lemma \ref{19071803}).}  
\begin{equation}
\text{det}\left(\begin{array}{cc}
a_{11} & b_{11}\\
a_{21} & b_{21} \\
\end{array}\right)\neq 0.
\end{equation}
Since the rank of 
\begin{equation}
\text{det}\left(\begin{array}{cc}
a_{11}+\tau b_{11} & a_{1i}+\tau b_{1i}\\
a_{21}+\tau b_{21} & a_{2i}+\tau b_{2i}\\
\end{array}\right)
\end{equation}
is $1$ for any $i$ ($2\leq i\leq n$), by Lemma \ref{17052301}, we get
\begin{equation}
\left(\begin{array}{cc}
a_{1i} & b_{1i}\\
a_{2i} & b_{2i}\\
\end{array}\right)
=\left(\begin{array}{cc}
l_ia_{11} & l_ib_{11} \\
l_ia_{21} & l_ib_{21} \\
\end{array}\right)\quad (l_i\in \mathbb{Q})
\end{equation}
for each $i$ ($2\leq i\leq n$). Now \eqref{19100903} is  
\begin{equation}
\begin{gathered}
a_{11}u_1+b_{11}(\tau u_1+h(u_1))+\sum_{i=2}^nl_i\big( a_{11}u_{i}+b_{11}(\tau u_{i}+h(u_i))\big)=0,\\
a_{21}u_1+b_{21}(\tau u_1+h(u_1))+\sum_{i=2}^nl_i\big(a_{21}u_{i}+b_{21}(\tau u_{i}+h(u_i))\big)=0
\end{gathered}
\end{equation}
and, since $a_{11}b_{21}-a_{21}b_{11}\neq 0$, it is equivalent to 
\begin{equation}\label{19100904}
\begin{gathered}
u_1+\sum_{i=2}^n l_iu_i=0,\\
\tau u_1+h(u_1)+\sum_{i=2}^n l_i(\tau u_i+h(u_i))=0.
\end{gathered}
\end{equation}
Since \eqref{19100904} defines an $(n-1)$-dimensional complex manifold, two equations in \eqref{19100904} are equivalent to each other. That is, if 
\begin{equation*}
u_1=-\sum_{i=2}^n l_iu_i,
\end{equation*}
then 
\begin{equation}\label{19100905}
\begin{gathered}
\tau \big(-\sum_{i=2}^n l_iu_i\big)+h\big(-\sum_{i=2}^n l_iu_i\big)+\sum_{i=2}^n l_i(\tau u_i+h(u_i))\\
=h\big(-\sum_{i=2}^n l_iu_i\big)+\sum_{i=2}^n l_ih(u_i)
\end{gathered}
\end{equation}
is the zero polynomial. Let 
\begin{equation*}
h(u_j)=\sum_{j=k}^{\infty} c_j u_j^k 
\end{equation*}
where $c_k\neq 0$ ($k\geq 3$). Then the part of homogeneous degree $k$ in \eqref{19100905} is 
\begin{equation}\label{19100906}
c_k\big(-\sum_{i=2}^n l_iu_i\big)^k+c_k\sum_{i=2}^n l_iu_i^k.
\end{equation}
It is easy to check \eqref{19100906} is the zero polynomial iff  
\begin{equation*}
l_i\neq 0
\end{equation*}
for at most one $i$ ($2\leq i\leq n$).
\begin{enumerate} 
\item If $l_i=0$ for all $i$ ($2\leq i\leq n$), then $\mathcal{N}_{\mathcal{X}^{(n)}}\cap H$ is contained in 
\begin{equation*}
M_1=L_1=1.
\end{equation*}
\item Suppose $l_i\neq 0$ for only one $i$ and, without loss of generality, we assume $l_2\neq 0$ and $l_i=0$ for all $3\leq i\leq n$. Then \eqref{19100906} is   
\begin{equation}\label{19100907}
c_k(-l_2u_2\big)^k+c_k l_2u_2^k.
\end{equation}
Since $k$ is odd, \eqref{19100907} is the zero polynomial iff $l_2=\pm 1$. In other words, $\mathcal{N}_{\mathcal{X}^{(n)}}\cap H$ is contained in either
\begin{equation*}
M_1=M_2, \quad L_1=L_2
\end{equation*}
or 
\begin{equation*}
M_1=M_2^{-1}, \quad L_1=L_2^{-1}.
\end{equation*}
This completes the proof. 
\end{enumerate}
\end{proof}

The following lemma is an analogue of Theorem \ref{19101107} and the strategy of the proof is the similar to that of Theorem \ref{19101107} using induction.

\begin{lemma}\label{19091508}
Let $\mathcal{X}^{(n)}$ be the same as above. For $|p_i|+|q_i|$ ($1\leq i\leq n$) sufficiently large, if the holonomies of a $(p_1/q_1, \dots, p_n/q_n)$-Dehn filling point of $\mathcal{X}^{(n)}$ are rationally dependent, then there exist $k,l$ ($1\leq k\neq l\leq n$) such that 
\begin{equation*}
p_k/q_k= p_l/q_l.
\end{equation*}
\end{lemma}
\begin{proof}
More generally, suppose there exists $\{P_i\}_{i\in \mathbb{N}}$ such that, for each $i\in \mathbb{N}$, $P_i$ is a $(p_{i1}/q_{i1},\dots, p_{i(n+m)}/q_{i(n+m)})$-Dehn filling point of $\mathcal{X}^{(n)}$ and 
\begin{equation}
P_i\rightarrow(1, \dots, 1) \quad \text{as} \quad i\rightarrow\infty.
\end{equation}
We further assume the holonomies of each $P_i$ are rationally dependent. By Theorem \ref{19090806}, there are finitely many algebraic subgroups $K_1, \dots, K_m$ such that
\begin{equation*}
P_i\subset \bigcup_{j=1}^m \mathcal{N}_{\mathcal{X}^{(n)}}\cap K_j
\end{equation*}
for $i$ sufficiently large. Without loss of generality, suppose $m=1$ and $K:=K_1$. We claim that, for each $i$ sufficiently large, there are $k,l$ ($1\leq k\neq l\leq n$) such that 
\begin{equation*}
p_{ik}/q_{ik}=p_{il}/q_{il}. 
\end{equation*} 
We prove it by induction on $n$. 

\begin{enumerate}
\item For $n=2$, since every anomalous subvariety is maximal and of codimension $1$, $\mathcal{N}_{\mathcal{X}^{(2)}}\cap K$ is contained in either 
\begin{equation*}
M_1=M_2, \quad L_1=L_2
\end{equation*}
or 
\begin{equation*}
M_1=M_2^{-1}, \quad L_1=L_2^{-1}
\end{equation*}
by Lemma \ref{19100901}. In either case, it implies 
\begin{equation*}
p_{i1}/q_{i1}=p_{i2}/q_{i2} 
\end{equation*}
for all $i$ sufficiently large.

\item Now let us assume $n\geq 3$ and the claim is true up to $n-1$. 
\begin{enumerate}
\item If $\mathcal{N}_{\mathcal{X}^{(n)}}\cap K$ is an anomalous subset of $\mathcal{N}_{\mathcal{X}^{(n)}}$ of codimension $1$, again by Lemma \ref{19100901}, we easily get the desired result. 

\item Now suppose $\mathcal{N}_{\mathcal{X}^{(n)}}\cap K$ is an anomalous subset of $\mathcal{N}_{\mathcal{X}^{(n)}}$ of codimension $\geq 2$. Without loss of generality, by applying Gauss elimination if necessary, assume $K$ is defined by 
\begin{equation}\label{19090808}
\begin{gathered}
M_1^{a_{11}}L_1^{b_{11}}\dots M_n^{a_{1n}}L_n^{b_{1n}}=1,\\
L_1^{b_{21}}M_2^{a_{22}}\dots M_n^{a_{2n}}L_n^{b_{2n}}=1,\\
M_2^{a_{32}}L_2^{b_{32}}\dots M_n^{a_{3n}}L_n^{b_{3n}}=1,\\
\cdots.
\end{gathered}
\end{equation}
Let 
\begin{equation*}
\mathcal{X}^{(n-1)}:=\mathcal{X}\times \cdots \times \mathcal{X}
\end{equation*}
be the product of $(n-1)$-copies of $\mathcal{X}$. Then the third equation in \eqref{19090808} implies the holonomies of a $(p_{i2}/q_{i2}, \dots, p_{in}/q_{in})$-Dehn filling point of $\mathcal{X}^{(n-1)}$ are rationally dependent. By induction, for $i$ sufficiently large, we get $k,l$ ($2\leq k\neq l\leq n$) such that 
\begin{equation*}
p_{ik}/q_{ik}=p_{il}/q_{il}.
\end{equation*}
This completes the proof. 
\end{enumerate}
\end{enumerate}
\end{proof}

Finally we prove Theorem \ref{19090606} which is restated below for reader's convenience.

\begin{theorem}
Let $\mathcal{M}$ be a $1$-cusped hyperbolic $3$-manifold having non-quadratic cusp shape and $\mathcal{M}_{p/q}$ be its $p/q$-Dehn filling. For $n>0$, if
\begin{equation*}
\{p_1/q_1, \dots, p_n/q_n\}
\end{equation*}
is $n$ distinct Dehn filling coefficients with $|p_i|+|q_i|$ ($1\leq i\leq n$) sufficiently large, then there are no $a_i\in \mathbb{Q}\backslash\{0\}$ ($1\leq i\leq n$) such that 
\begin{equation*}
a_1\text{pvol}_{\mathbb{C}}\;\mathcal{M}_{p_1/q_1}+\cdots+a_n\text{pvol}_{\mathbb{C}}\;\mathcal{M}_{p_n/q_n}\equiv 0 \mod i\pi^2\mathbb{Z}.
\end{equation*}  
\end{theorem}

\begin{proof}
On the contrary, suppose there exists a family  
\begin{equation}\label{19101402}
\{\mathcal{M}_{p_{i1}/q_{i1}}, \dots, \mathcal{M}_{p_{in}/q_{in}}\}_{i\in \mathbb{N}}
\end{equation}
of infinitely many $n$ distinct hyperbolic Dehn fillings of $\mathcal{M}$ such that 
\begin{equation}\label{19091513}
\sum_{k=1}^n a_{ik} \text{pvol}\;\mathcal{M}_{p_{ik}/q_{ik}}\equiv 0\mod i\pi^2\mathbb{Z}\quad (a_{ik}\in \mathbb{Z}\backslash\{0\}).
\end{equation}
For each $k$ ($1\leq k\leq n$), we further suppose  
\begin{equation}\label{19100704}
\begin{gathered}
\lim_{i\rightarrow \infty} |p_{ik}|+|q_{ik}|=\infty.
\end{gathered}
\end{equation}
Since 
\begin{equation*}
\text{pvol}\;(\mathcal{M}_{p_{ik}/q_{ik}})\equiv\text{vol}\;\mathcal{M}-\frac{\pi}{2}\log\;t_{p_{ik}/q_{ik}} \mod i\pi^2\mathbb{Z}
\end{equation*}
where $t_{p_{ik}/q_{ik}}$ is the holonomy of $\mathcal{M}_{(p_{ik}/q_{ik})}$, \eqref{19091513} implies 
\begin{equation}
\sum_{k=1}^n a_{ik}\big(\text{vol}\;\mathcal{M}-\frac{\pi}{2}\log\;t_{p_{ik}/q_{ik}}\big)\equiv 0 \mod i\pi^2\mathbb{Z}
\end{equation}
and so  
\begin{equation}\label{19100701}
\text{vol}\;\mathcal{M}\sum_{k=1}^n a_{ik}-\frac{\pi}{2}\sum_{k=1}^n \big(a_{ik}\log\;t_{p_{ik}/q_{ik}}\big)\equiv 0\mod i\pi^2\mathbb{Z}.
\end{equation}
\begin{enumerate}
\item First if  
\begin{equation}\label{19100702}
\sum_{k=1}^n a_{ik}=0,
\end{equation}
then
\begin{equation}\label{19100101}
\begin{gathered}
\sum_{k=1}^n a_{ik}\log\;t_{p_{ik}/q_{ik}}\equiv 0\mod 2\pi i\mathbb{Z}
\end{gathered}
\end{equation}
and this is equivalent to 
\begin{equation}\label{19101401}
\begin{gathered}
\prod_{k=1}^n t^{a_{ik}}_{p_{ik}/q_{ik}}=1.
\end{gathered}
\end{equation}
Let $\mathcal{X}$ be the A-polynomial of $\mathcal{M}$ and 
\begin{equation*}
\mathcal{X}^{(n)}:=\mathcal{X}\times \cdots \times \mathcal{X}.
\end{equation*}
For each $i\in \mathbb{N}$, we consider 
\begin{equation}\label{19102102}
\{t_{p_{i1}/q_{i1}}, \dots, t_{p_{in}/q_{in}}\}
\end{equation}
as the set of the holonomies of a $(p_{i1}/q_{i1}, \dots, p_{in}/q_{in})$-Dehn filling point of $\mathcal{X}^{(n)}$. For each $i\in \mathbb{N}$, since the elements in \eqref{19102102} are rationally dependent by \eqref{19101401}, if $i$ is sufficiently large, we get $k, l$ ($1\leq k\neq l\leq n$) such that    
\begin{equation*}
p_{ik}/q_{ik}=p_{il}/q_{il}
\end{equation*}
by Lemma \ref{19091508}. But this contradicts the assumption that \eqref{19101402} is a set of infinitely many $n$-distinct Dehn fillings for each $i\in \mathbb{N}$.  

\item Now suppose 
\begin{equation}
\sum_{k=1}^n a_{ik}\neq 0\quad (i\in \mathbb{N})
\end{equation}
and pick $i,j\in \mathbb{N}$ sufficiently large such that\footnote{Note that there are infinitely many such pairs by the assumption \eqref{19100704}.} 
\begin{equation}\label{19092204}
\{p_{i1}/q_{i1},\dots, p_{in}/q_{in}\}\neq \{p_{j1}/q_{j1},\dots, p_{j1}/q_{j1}\}.  
\end{equation}
By \eqref{19100701}, 
\begin{equation}
\begin{gathered}
\Big(\sum_{k=1}^n a_{jk}\Big)\Big(\sum_{k=1}^n a_{ik}\log\;t_{p_{ik}/q_{ik}}\Big)
\equiv\Big(\sum_{k=1}^n a_{ik}\Big)\Big(\sum_{k=1}^n a_{jk}\log\;t_{p_{jk}/q_{jk}}\Big)\mod 2\pi i\mathbb{Z}
\end{gathered}
\end{equation}
 and this implies
\begin{equation}
\begin{gathered}
\Big(\prod_{k=1}^n t^{a_{ik}}_{p_{ik}/q_{ik}}\Big)^{\sum_{k=1}^n a_{jk}}=\Big(\prod_{k=1}^n t^{a_{jk}}_{p_{jk}/q_{jk}}\Big)^{\sum_{k=1}^n a_{ik}}.
\end{gathered}
\end{equation}
Similar to the previous case, let 
\begin{equation*}
\mathcal{X}^{(2n)}
\end{equation*}
be the product of $2n$-copies of $\mathcal{X}$ and consider 
\begin{equation*}
\{t_{p_{i1}/q_{i1}}, \dots, t_{p_{in}/q_{in}}, t_{p_{j1}/q_{j1}}, \dots, t_{p_{jn}/q_{jn}}\}
\end{equation*}
as the set of the holonomies of a $(p_{i1}/q_{i1}, \dots, p_{in}/q_{in}, p_{in}/q_{in}, \dots, p_{jn}/q_{jn})$-Dehn filling point of $\mathcal{X}^{(2n)}$. Again by Lemma \ref{19091508}, we get $k, l$ ($1\leq k, l\leq n$) such that 
\begin{equation*}
p_{ik}/q_{ik}=p_{jl}/q_{jl}.
\end{equation*}
But this contradicts the assumption \eqref{19092204}. This completes the proof. 
\end{enumerate}
\end{proof}

\vspace{5 mm}
Department of Mathematics, POSTECH\\
77 Cheong-Am Ro, Pohang, South Korea\\
\\
\emph{Email Address}: bogwang.jeon@postech.ac.kr

\newpage
\begin{center}
\textbf{\large Appendix A. A simple proof of Corollary \ref{19080902} \\
by Ian Agol}
\end{center}

In this note, we give a simple proof of Corollary \ref{19080902}. We only prove it for the single cusped case, since it is just
a matter of bookkeeping to generalize it to multi-cusped cases. 

Consider an orientable hyperbolic 3-manifold $\mathcal{M}$ with
a single cusp. Take a maximal horoball neighborhood $\mathcal{H}$ of
this cusp. Choosing longitude and meridian coordinates for the
torus $\partial\mathcal{H}$, let $\mathcal{M}(p/q)$ denote the Dehn filling of $\mathcal{M}$ of
slope $p/q$. For $|p|+|q|$ sufficiently large, $\mathcal{M}(p/q)$ will be obtained
by hyperbolic Dehn filling on $\mathcal{M}$ in the sense of Thurston, and hence the core of the
Dehn filling will be a geodesic. Let $t_{p/q}$ denote the holonomy of this geodesic. 

\begin{theorem}
Let $\mathcal{M}$ be a 1-cusped hyperbolic manifold whose cusp
does not admit a non-trivial (orientation-preserving) isometry. Then for $|p|+|q|$ and $|p'|+|q'|$ sufficiently large,
 $$t_{p/q}=t_{p'/q'}$$
if and only if $p/q=p'/q'$.

\end{theorem}
\begin{proof}
The assumption on the cusp $\mathcal{H}$
means that any orientation-preserving isometry of $\partial{\mathcal{H}}$ 
is a rotation or elliptic involution. Such isometries always preserve isotopy
classes of (unoriented) curves. 

For $|p|+|q|$ large enough, let $\gamma_{p/q}$ be the core geodesic
of the Dehn filling $\mathcal{M}(p/q)$, with holonomy $t_{p/q}$. Then
for $|p|+|q|$ large enough, there will be a maximal embedded tubular
neighborhood $\mathcal{T}_{p/q}$ of radius $R_{p/q}$ about $\gamma_{p/q}$,
such that $\mathcal{T}_{p/q}$ converges geometrically to $\mathcal{H}$
as $|p|+|q| \to \infty$ in the pointed Gromov-Hausdorff topology, where
by convention we take basepoints to lie in $\partial{T}_{p/q}$ and $\partial\mathcal{H}$. 
Since $\partial \mathcal{H}$
does not admit a non-trivial orientation-preserving isometry, then for $|p|+|q|$
large enough, $\partial\mathcal{T}_{p/q}$ will also not admit a non-trivial orientation
preserving isometry (since the moduli of tori admitting an isometry are two isolated
points in moduli space, the square and hexagonal tori). 

Now, suppose $t_{p/q}=t_{p'/q'}$ for $|p|+|q|$ and $|p'|+|q'|$ large enough. 
Then for any $\epsilon>0$, we may choose them large enough so that  $|Vol(\mathcal{T}_{p/q}) -Vol(\mathcal{H})|<\epsilon$
and $|Vol(\mathcal{T}_{p'/q'})-Vol(\mathcal{H})| <\epsilon$. 
Then for any $\delta > 0$ we may choose the coefficients  large enough so
that $|R_{p/q}-R_{p'/q'}|<\delta$, since $Vol(\mathcal{T}_{p/q})$ is determined
as a continuous function of $t_{p/q}$ and $R_{p/q}$. Suppose (WLOG) that $R_{p/q}< R_{p'/q'}$. 
Let $\mathcal{T}'_{p'/q'}$ be the tubular neighborhood of radius $R_{p/q}$ about
$\gamma_{p'/q'}$ in $\mathcal{M}_{p'/q'}$, which is therefore isometric to 
$\mathcal{T}_{p/q}$. Moreover, $\partial\mathcal{T}'_{p'/q'}$ is arbitrarily close
to $\partial\mathcal{T}_{p'/q'}$ in the Gromov-Hausdorff topology. Hence, we have two
tori which are arbitrarily close to $\partial \mathcal{H}$ in the Gromov-Hausdorff topology
which are isometric. Since no torus in a neighborhood of $\partial \mathcal{H}$
in moduli space admits a non-trivial symmetry, this isometry must be trivial,
and hence map the curve of slope $p/q$ to the curve of slope $p'/q'$,
and hence $p/q=p'/q'$. 
\end{proof}

{\bf Remark:} The assumption that the cusps admit no symmetries is equivalent
to saying that the cusp shape does not correspond to the square torus (obtained
by identifying opposite sides of a square) or a hexagonal torus (obtained by identifying
opposite sides of a regular hexagon). In turn, this implies that the cusp shape does
not lie in $\mathbb{Q}(i)$ or $\mathbb{Q}(\sqrt{-3})$ respectively. \\
%{\bf Remark:} The above theorem generalizes to multiple cusps, it is just
%a matter of bookkeeping. \\
\\
University of California, Berkeley \\
970 Evans Hall \#3840 \\ 
Berkeley, CA 94720-3840\\               
\\
\emph{Email Address}: ianagol@math.berkeley.edu


\begin{thebibliography}{99}
\bibitem{bp} R. ~Benedetti, C. ~Petronio, \emph{Lectures on Hyperbolic Geometry}, Springer (1992).
\bibitem{BMZ0} E. ~Bombieri, D. ~Masser, U. ~Zannier, \emph{Intersecting a plane with algebraic subgroups of multiplicative groups}, Ann. Scuola Norm. Sup. Pisa Cl. Sci. (2008), 51-80. 
\bibitem{za} E. ~Bombieri, D. ~Masser, U. ~Zannier, \emph{Anomalous subvarieties-structure theorems and applications}, IMRN (2007), 1-33.
\bibitem{BMZ1}  E. ~Bombieri, D. ~Masser, U. ~Zannier, \emph{Intersection of complex varieties with tori}, Acta Arith. (2008), 309-323.
\bibitem{bg} E. ~Bombieri, W. ~Gubler, \emph{Heights in Diophantine Geometry}, Cambridge University Press (2006).
%\bibitem{ccgls} D. ~Cooper, M. ~Culler, H. ~Gillett, D.~D. ~Long, P.~B. ~Shalen, \emph{Plane curves associated to character varieties of knot complements}, Invent. math. 118 (1994) 47-84.

\bibitem{FPS} D.~Futer, J.~Purcell, S. ~Schleimer, \emph{Effective bilipschitz bounds on drilling and filling}, arxiv.org/abs/1907.13502.

\bibitem{gromov} M. ~Gromov, \emph{Hyperbolic manifolds (according to Thurston and Jørgensen)}, Bourbaki Seminar (1979/80), 40-53, Lecture Notes in Math. 842, Springer, Berlin-New York, 1981.
\bibitem{hab} P. ~Habegger, \emph{On the bounded height conjecture}, IMRN (2009), 860-886.
\bibitem{jeon1} B. ~Jeon, \emph{Hyperbolic $3$-manifolds of bounded volume and trace field degree}, arxiv.org/abs/1305.0611.  
\bibitem{jeon2} B. ~Jeon, \emph{Hyperbolic $3$-manifolds of bounded volume and trace field degree II}, arxiv.org/abs/1409.2069. 
\bibitem{jeon4} B. ~Jeon, \emph{The Zilber-Pink conjecture and the generalized Cosmetic Surgery Conjecture}, arxiv.org/abs/1801.07819. 
\bibitem{jeon3} B. ~Jeon, \emph{On the number of Dehn fillings of a given volume}, arxiv.org/abs/1812.04788. 
\bibitem{rigidity} W. ~Neumann, A.~Reid, \emph{Rigidity of cusps in deformations of hyperbolic 3-orbifolds}, Math. Ann. (1993), 223-237.
\bibitem{nz} W. ~Neumann, D. ~Zagier, \emph{Volumes of hyperbolic three-manifolds}, Topology (1985), 307-332.
\bibitem{thu} W.~Thurston, \emph{The Geometry and Topology of 3-manifolds}, Princeton University Mimeographed Notes (1979).
\bibitem{thu2} W.~Thurston, \emph{Three dimensional manifolds, Kleinian groups and hyperbolic geometry}, Bull. Amer. Math. Soc. (1982). 
\bibitem{Y} T.~Yoshida, \emph{The $\eta$-invariant of hyperbolic $3$-manifolds}, Invent. Math. (1985), 473-514.
\bibitem{zan} U.~Zannier, \emph{Some problems of unlikely intersections in Arithmetic and Geometry}, Princeton University Press (2012). 
\end{thebibliography}
\end{document}